\documentclass[12pt,reqno]{amsart}
\usepackage{amsmath,amssymb,extarrows}
\usepackage{url}
\usepackage{tikz,enumerate}
\usepackage{diagbox}
\usepackage{appendix}
\usepackage{epic}

\usepackage{float} 

\usepackage{cite}

\usepackage{hyperref}
\usepackage{array}

\usepackage{booktabs}

\allowdisplaybreaks[4]

\newtheorem{theorem}{Theorem}
\newtheorem{corollary}[theorem]{Corollary}
\newtheorem{conj}[theorem]{Conjecture}
\newtheorem{lemma}[theorem]{Lemma}
\newtheorem{prop}[theorem]{Proposition}
\theoremstyle{definition}

\theoremstyle{remark}
\newtheorem{rem}{Remark}
\numberwithin{equation}{section}
\numberwithin{theorem}{section}
\numberwithin{defn}{section}

\begin{document}
\title[Modularity of Nahm Sums for the Tadpole Diagram]
{Modularity of Nahm Sums for the Tadpole Diagram}

\author{Antun Milas and Liuquan Wang}
\address{Department of Mathematics and Statistics, University at Albany (SUNY), Albany, NY 12222,
United States}
\email{amilas@math.albany.edu}
\address{School of Mathematics and Statistics, Wuhan University, Wuhan 430072, Hubei, People's Republic of China}
\email{wanglq@whu.edu.cn;mathlqwang@163.com}




\begin{abstract}
We prove Rogers-Ramanujan type identities for the Nahm sums associated with the tadpole Cartan matrix of rank $3$. These identities reveal the modularity of these sums, and thereby we confirm a conjecture of Penn, Calinescu and the first author in this case. We show that these Nahm sums together with some shifted sums can be combined into a vector-valued modular function on the full modular group.  We also present some conjectures for a general rank.
\end{abstract}

\maketitle

\section{Introduction and Main Results}
The famous Rogers-Ramanujan identities state that
\begin{align}
\sum_{n=0}^\infty \frac{q^{n^2}}{(q;q)_n}=\frac{1}{(q,q^4;q^5)_\infty}, \label{RR1}\\
\sum_{n=0}^\infty \frac{q^{n^2+n}}{(q;q)_n}=\frac{1}{(q^2,q^3;q^5)_\infty}, \label{RR2}
\end{align}
where (here and throughout this paper) we always assume $|q|<1$ and use standard $q$-series notations:
\begin{align}
(a;q)_0:=1, \quad (a;q)_n:=\prod\limits_{k=0}^{n-1}(1-aq^k), \quad (a;q)_\infty :=\prod\limits_{k=0}^\infty (1-aq^k),  \\
(a_1,\dots,a_m;q)_n:=(a_1;q)_n\cdots (a_m;q)_n, \quad n\in \mathbb{N}\cup \{\infty\}.
\end{align}
When the base $q$ is clear from the context,  occasionally we omit it and simply write $(a;q)_n$ as $(a)_n$ ($n\in \mathbb{N}\cup \{\infty\})$.

The Rogers-Ramanujan identities first appeared in the 1894 paper of Rogers \cite{Rogers1894} and were later rediscovered by Ramanujan before 1913. Besides \eqref{RR1} and \eqref{RR2}, Rogers \cite[pp.\ 330-332]{Rogers1894} also proved some similar sum-product identities such as 
\begin{align}
\sum_{n=0}^\infty \frac{q^{n^2}}{(q;q)_{2n}}&=\frac{(q^2,q^8,q^{10};q^{10})_\infty (q^6,q^{14};q^{20})_\infty}{(q;q)_\infty}, \quad  \label{S79}\\
\sum_{n=0}^\infty \frac{q^{n^2+n}}{(q;q)_{2n}}&=\frac{(q,q^9,q^{10};q^{10})_\infty (q^8,q^{12};q^{20})_\infty}{(q;q)_\infty}, \quad  \label{S99} \\
\sum_{n=0}^\infty \frac{q^{n^2+n}}{(q;q)_{2n+1}}&=\frac{(q^3,q^7,q^{10};q^{10})_\infty (q^4,q^{16};q^{20})_\infty}{(q;q)_\infty},  \quad  \label{S94}
\end{align}
Later Rogers \cite[p.\ 330 (3), 2nd Eq.]{Rogers1917}) proved another companion identity:
\begin{align}
\sum_{n=0}^\infty \frac{q^{n^2+2n}}{(q;q)_{2n+1}}=\frac{(q^4,q^6,q^{10};q^{10})_\infty (q^2,q^{18};q^{20})_\infty}{(q;q)_\infty}. \label{Rogers-1}
\end{align}


The Rogers-Ramanujan identities also serve as important examples for close relations between $q$-series and modular forms. The product sides are essentially reciprocals of some generalized Dedekind eta functions (see \eqref{general-eta}) and hence are modular functions. This is not observable from the sum sides, which is a basic $q$-hypergeometric series. A natural question is to ask when does a basic $q$-hypergeometric series become a modular form. In particular, a famous problem of Nahm is to determine for which positive definite $r\times r$ rational matrix $A$, $r$-dimensional rational vector $B$, and a rational scalar $C$ such that
$$f_{A,B,C}(q):=\sum_{n=(n_1,\dots,n_r)^\mathrm{T} \in (\mathbb{Z}_{\geq 0})^r} \frac{q^{\frac{1}{2}n^\mathrm{T} An+n^\mathrm{T} B+C}}{(q;q)_{n_1}\cdots (q;q)_{n_r}}$$
is a modular form. Such $(A,B,C)$ is called as a rank $r$ {\em modular triple}. The series $f_{A,B,C}(q)$ is therefore referred as {\em Nahm sums}. 

Several important families of $q$-series identities (such as Andrews-Gordon identities) can be also studied using vertex operators and representation of infinite-dimensional Lie algebras. This approach, pioneered by Lepowsky and Wilson in 1980s \cite{Lepowsky-Wilson,Lepowsky-Wilson-1985}, was one of the starting points in the development of vertex operator algebras and an important ingredient in the development of 2-dimensional conformal field theory (CFT) in physics. In this setup, the graded dimension obtained from a combinatorial bases of modules can be often interpreted as a Nahm sums. 
The Nahm sums that are relevant for rational CFT almost always take form with $A=G \otimes G'^{-1}$ where $G$ and $G'$ are ADET type Cartan matrices, and all such Nahm sums matrices are expected to give modular functions (with appropriate $B$ and $C$). Interestingly, {\em any} Nahm sum $f_{A,0,0}$ can be interpreted as the graded dimension of a special vertex algebra called {\em principal subspace} \cite{Milas-Penn}. 


Zagier \cite{Zagier} studied Nahm's problem and made significant progress when the rank $r\leq 3$. In particular, he proved that there are exactly seven rank one modular triples. In the rank two and three cases, Zagier provided a number of conjectural modular triples. Most of the rank two examples have been confirmed in the literature. For example, Vlasenko and Zwegers \cite{VZ} confirmed one modular triple in Zagier's list and discovered some new examples. Recently, the second author \cite{Wang-Two} confirmed more examples in Zagier's list. As a consequence, among the eleven rank two examples discovered by Zagier \cite[Table 2]{Zagier}, only one example is unproven. In the rank three case, Zagier \cite[Table 3]{Zagier} provided a list of twelve possible modular triples and proved three of them. The remaining nine examples were confirmed by the second author \cite{Wang-Three}.

In this paper, we are mainly concerned with the modularity of the Nahm sums associated with the tadpole diagram $T$. The tadpole Nahm sums were considered by Penn, Calinescu, and the first author in their work on twisted modules of principal subspace vertex algebras \cite{CMP}.
 Given a positive integer $r$, let $T_r$ be the tadpole Cartan matrix, that is, $T_r=(a_{ij})_{r\times r}$ such that 
 \begin{align*}
 a_{rr}=1, a_{ii}=2, 1 \leq i \leq r-1, a_{ij}=-1 ~~ (|i-j|=1), \quad \text{and} \quad a_{ij}=0 \quad \text{otherwise}.
 \end{align*}
 We let 
 $$\chi_0(x_1,\dots,x_r)=\chi_0(x_1,\dots,x_r;q):=\sum_{n=(n_1,\dots,n_r) \in \mathbb{Z}_{\geq 0}^r} \frac{q^{\frac12 {n}^\mathrm{T}T_r n} x_1^{n_1} \cdots x_n^{n_r}}{(q)_{n_1} \cdots (q)_{n_r}}$$
 be a generalized tadpole Nahm sum. It is easy to see that $$q^C\chi_0(q^{B_1},\dots,q^{B_r})=f_{A,B,C}(q)$$ 
 is the Nahm sum with $A=T_r$ and $B=(B_1,\dots,B_r)$.
There is only one standard $A_{2n}^{(2)}$-module of level one (up to isomorphism) and thus only one principal subspace of level $1$, corresponding to the unique standard level one module, 
whose character is $\chi_0(1,1,\dots,1)$, i.e. $x_i=1$ for all $i$.
In \cite{CMP} the so-called shifted characters $\chi_i$ were introduced by specializing $x_i=q$ and $x_j=1$ for $j \neq i$. Penn, Calinescu, and the first author \cite{CMP} stated a conjecture concerning the modularity of the characters $\chi_0(x_1,\dots,x_r)$.
\begin{conj}\label{cmp}
(Cf.\ \cite[Conjecture 1]{CMP}.)  The character $q^a \chi_0(1,\dots,1)$ is modular for some rational number $a$.
\end{conj}
The rank two case ($r=2$) was proved in \cite{CMP}.
The main goal of this paper is to address the conjecture for $r=3$ case. In this case, we write explicitly
\begin{align}\label{eq-chi}
\chi_0(x_1,x_2,x_3)=\chi_0(x_1,x_2,x_3;q)=\sum_{i,j,k \geq 0} \frac{q^{i^2+j^2+ \frac12 k^2-ij -jk}x_1^{i} x_2^{j} x_3^{k}}{(q)_{i}(q)_{j}(q)_{k}}.
\end{align}
We record four shifted $q$-characters that are of interest here:
\begin{align*}
&F_1(q):=\chi_0(1,1,1), \quad F_2(q):=\chi_0(1,1,q^{\frac12}), \\ &F_3(q):=\chi_0(q,q^{-1},q^{\frac12}), \quad F_4(q):=\chi_0(q^{-1},q,1).
\end{align*}
We are mainly concerned with  modular properties of these sums. We first write down the tadpole Cartan matrix $T_3$ and its inverse:
$$T_3=\begin{pmatrix}
2 & -1 & 0 \\ -1 & 2 & -1 \\ 0 & -1 & 1
\end{pmatrix}, \quad T_3^{-1}=\begin{pmatrix}
1 & 1 & 1 \\ 1 & 2 & 2 \\ 1 & 2 & 3
\end{pmatrix}.
$$
Note that $T_3^{-1}$ is the matrix part of the sixth example in Zagier's list \cite[Table 3]{Zagier} (see also \cite[Example 6]{Wang-Three}) with the first row/column and the third row/column interchanged. Zagier stated six possible modular triples with $T_3^{-1}$ as the matrix part. The vector parts are
\begin{align}
B\in \left\{ \begin{pmatrix} 0 \\ 0 \\ 0 \end{pmatrix},   \begin{pmatrix} 1/2 \\ 1 \\ 3/2 \end{pmatrix}, \begin{pmatrix} 1/2 \\ 0 \\ 1/2 \end{pmatrix}, \begin{pmatrix} 0 \\ 1 \\ 1 \end{pmatrix}, \begin{pmatrix}
-1/2 \\ 0 \\ -1/2 \end{pmatrix},  \begin{pmatrix} 1/2 \\ 1 \\ -1/2 \end{pmatrix}\right\}.
\end{align}
Here the first and the third entries in each vector have been interchanged and we have reordered these vectors. 
Zagier \cite[p.\ 50]{Zagier} conjectured that there are some duality between modular triples. Namely, he mentioned that when $(A,B,C)$ is a rank $r$ modular triple, then it is likely that
$$(A^\star, B^\star, C^\star)=(A^{-1},A^{-1}B,\frac{1}{2}B^\mathrm{T} A^{-1}B-\frac{r}{24}-C)$$
is also a rank $r$ modular triple. This motivates us to consider the dual cases to the six modular triples related to $T_3^{-1}$. To be specific, the dual vectors are 
\begin{align*}
    B^\star \in \left\{\begin{pmatrix}
0 \\ 0 \\ 0
    \end{pmatrix},  \begin{pmatrix}
        0 \\ 0 \\ 1/2 
    \end{pmatrix}, \begin{pmatrix}
        1 \\ -1  \\ 1/2 
    \end{pmatrix}, \begin{pmatrix}
        -1 \\ 1 \\ 0
    \end{pmatrix}, \begin{pmatrix}
        -1 \\ 1 \\ -1/2
    \end{pmatrix},  \begin{pmatrix}
        -2 \\ 2 \\ -1/2
    \end{pmatrix} \right\}.
\end{align*}
We find that they are indeed modular triples for suitable $C$. This is a consequence of the following set of Rogers-Ramanujan type identities. Before we state them, we introduce the compact notations
$$J_m:=(q^m;q^m)_\infty, \quad J_{a,m}:=(q^a,q^{m-a},q^m;q^m)_\infty.$$
\begin{theorem} \label{main2}
We have
\begin{align}
\sum_{i,j,k\geq 0} \frac{q^{2i^2+2j^2+k^2-2ij-2jk}}{(q^2;q^2)_i(q^2;q^2)_j(q^2;q^2)_k}&=\frac{J_4^5J_{40}}{J_1J_2^2J_8^2J_{8,40}}+2q\frac{J_8^2J_{6,20}J_{8,40}}{J_1J_4^2J_{40}}, \label{sum-id-1} \\
\sum_{i,j,k\geq 0} \frac{q^{i^2+j^2+\frac{1}{2}(k^2+k)-ij-jk}}{(q;q)_i(q;q)_j(q;q)_k}
&=\frac{J_2^6J_{1,10}J_{8,20}}{J_1^5J_4^2J_{20}}+2q\frac{J_4^2J_{4,10}J_{2,20}}{J_1^3J_{20}}, \label{sum-id-2} \\
\sum_{i,j,k\geq 0} \frac{q^{i^2+j^2+\frac{1}{2}(k^2+k)-ij-jk+i-j}}{(q;q)_i(q;q)_j(q;q)_k}
&=2\frac{J_2J_4^2J_{20}}{J_1^3J_{4,20}}+\frac{J_2^6J_{3,10}J_{4,20}}{J_1^5J_4^2J_{20}},  \label{sum-id-3} \\
\sum_{i,j,k\geq 0} \frac{q^{2i^2+2j^2+k^2-2ij-2jk-2i+2j}}{(q^2;q^2)_i(q^2;q^2)_j(q^2;q^2)_k}
&=2\frac{J_8^2J_{2,20}J_{16,40}}{J_1J_4^2J_{40}}+q\frac{J_4^4J_{8,20}J_{4,40}}{J_1J_2^2J_8^2J_{40}},  \label{sum-id-4} \\
\sum_{i,j,k\geq 0} \frac{q^{i^2+j^2+\frac{1}{2}(k^2-k)-ij-jk-i+j}}{(q;q)_i(q;q)_j(q;q)_k}&=4\frac{J_2J_4^2J_{20}}{J_1^3J_{4,20}}+2\frac{J_2^6J_{3,10}J_{4,20}}{J_1^5J_4^2J_{20}},  \label{sum-id-5} \\
\sum_{i,j,k\geq 0} \frac{q^{i^2+j^2+\frac{1}{2}(k^2-k)-ij-jk-2i+2j}}{(q;q)_i(q;q)_j(q;q)_k}
&=2q^{-1}\frac{J_2^6J_{1,10}J_{8,20}}{J_1^5J_4^2J_{20}}+4\frac{J_4^2J_{4,10}J_{2,20}}{J_1^3J_{20}}.  \label{sum-id-6}
\end{align}
\end{theorem}
Note that the left sides of \eqref{sum-id-1}--\eqref{sum-id-6} are the Nahm sums
\begin{align*}
F_1(q^2),  F_2(q), F_3(q), F_4(q^2), \chi_0(q^{-1},q,q^{-\frac{1}{2}};q) ~~\text{ and}~~ \chi_0(q^{-2},q^2,q^{-\frac{1}{2}};q),
\end{align*}
respectively.
In light of the modularity of the functions $J_m$ and $J_{a,m}$, it is easy to verify (e.g., using the algorithm in \cite{Garvan-Liang}) that the Nahm sums in \eqref{sum-id-1}--\eqref{sum-id-6} are all modular functions after multiplying a factor $q^C$ with $C$ being 
$$-\frac{7}{40},\quad \frac{1}{40}, \quad \frac{9}{40}, \quad \frac{17}{40}, \quad \frac{9}{40}, \quad \frac{41}{40},$$
respectively. In particular, we see that the identity \eqref{sum-id-1} confirms Conjecture \ref{cmp} in the case $r=3$. Interestingly, there are essentially four different modular functions for these six Nahm sums. In fact, comparing the right sides of \eqref{sum-id-3} and \eqref{sum-id-5}, and \eqref{sum-id-2} and \eqref{sum-id-6},  we see that
\begin{align}
\chi_0(q^{-1},q,q^{-\frac{1}{2}};q)&=2\chi_0(q,q^{-1},q^{\frac{1}{2}};q), \label{Nahm-relation-1} \\
\chi_0(q^{-2},q^2,q^{-\frac{1}{2}};q)&=2q^{-1}\chi_0 (1,1,q^{\frac{1}{2}};q). \label{Nahm-relation-2}
\end{align}
This is not obvious from their original definitions. We will explain this in the proof of this theorem.

Next, we will study the tadpole Nahm sums from the point of vector-valued modular forms. By doing so, we will be able to see their modular transformation properties more clearly .

Let $\mathbb{H}$ denote the upper half complex plane. Throughout this paper we denote $q=e^{2 \pi i \tau}$, $\tau \in \mathbb{H}$. Since the shifted characters $\chi_0(1,1,1)$ and $\chi_0(q^{-1},q,1)$ have $q$-powers in $\mathbb{Z}+\frac12$ it is convenient to consider
\begin{align}
F_5(q):=\chi_0(1,1,1)|_{\tau \to \tau+1}&=\sum_{i,j,k \geq 0} (-1)^k \frac{q^{ i^2+j^2+\frac12 {k^2}-ij-jk}}{(q)_i(q)_j(q)_k}, \label{F5-defn}\\
F_6(q):=\chi_0(q^{-1},q,1)|_{\tau \to \tau+1}&=\sum_{i,j,k \geq 0} (-1)^k \frac{q^{ i^2+j^2+\frac12 {k^2}-ij-jk-i+j}}{(q)_i(q)_j(q)_k}. \label{F6-defn}
\end{align}
For $1 \leq i \leq 6$ we define $\tilde{F}_i(q):=q^{\lambda_i}{F}_i(q)$ where
$$\lambda_1=-\frac{7}{80}, ~~\lambda_2=\frac{1}{40}, ~~
 \lambda_3=\frac{9}{40}, ~~ \lambda_4=\frac{17}{80}, ~~ \lambda_5=-\frac{7}{80}, ~~\lambda_6=\frac{17}{80}.$$

We define Weber's modular functions \cite{We}:
\begin{align}\label{Weber-defn}
\mathfrak{f}(\tau):=q^{-1/48} (-q^{1/2};q)_\infty, \ \  \mathfrak{f}_1(\tau):=q^{-1/48} (q^{1/2};q)_\infty, \ \ \mathfrak{f}_2(\tau):=q^{1/24}(-q;q)_\infty.
\end{align}
For $k>0$ and $j\in \mathbb{Q}$ we let
\begin{align}
(\partial \Theta)_{j,k}(\tau) &:=\sum_{n \in \mathbb{Z}} (2kn+j)q^{(2kn+j)^2/(4k)}, \label{partial-Theta-defn}\\
 (\partial G)_{j,k}(\tau)&:=\sum_{n \in \mathbb{Z}} (-1)^n (2kn+j)q^{(2kn+j)^2/(4k)}. \label{partial-G-defn}
\end{align}
When $k,j \in \mathbb{Q}$, these are essentially Jacobi theta series of weight $3/2$.
\begin{theorem} \label{main} We have the following $q$-identities:
\begin{align*}
\tilde{F}_1(q)& =\frac{\mathfrak{f}(\tau)^3}{\eta(\tau)^3} (\partial \Theta)_{1,\frac52}(\tau), \\
\tilde{F}_2(q)&=2\frac{\mathfrak{f}_2(\tau)^3}{\eta(\tau)^3} (\partial \Theta)_{\frac12,\frac52}(\tau), \\
\tilde{F}_3(q)&=2\frac{\mathfrak{f}_2(\tau)^3}{\eta(\tau)^3} (\partial \Theta)_{\frac32,\frac52}(\tau), \\
\tilde{F}_4(q)&=\frac{\mathfrak{f}(\tau)^3}{\eta(\tau)^3} (\partial \Theta)_{2,\frac52}(\tau), \\
\tilde{F}_5(q)& =\frac{\mathfrak{f}_1(\tau)^3}{\eta(\tau)^3} (\partial G)_{1,\frac52}(\tau), \\
\tilde{F}_6(q)&=\frac{\mathfrak{f}_1(\tau)^3}{\eta(\tau)^3} (\partial G)_{2,\frac52}(\tau).
\end{align*}
In particular, $\tilde{F}_i(q)$ are modular functions on some congruence subgroups of $\mathrm{SL}(2,\mathbb{Z})$.
Moreover, $( \tilde{F}_i(q) )_{1 \leq i \leq 6}$ transforms as a vector valued modular function on $\mathrm{SL}(2,\mathbb{Z})$.
\end{theorem}
This in particular proves and extends Conjecture \ref{cmp} for $r=3$.

The rest of this paper is organized as follows. In Section \ref{sec-identities} we present  a proof for Theorem \ref{main2}. The idea is to use constant term method to reduce triple sums to some single sums and use Rogers' identities \eqref{S79}--\eqref{Rogers-1}. In Section \ref{sec-proofs} we give two different proofs for Theorem \ref{main}. Finally, in Section \ref{sec-general} we prove that some other Nahm sums associated with the tadpole Cartan matrix are not modular, and we give a general conjecure on modular Nahm sums.

\begin{rem} \label{new-basis} It will be useful to observe another basis of $V={\rm Span}\{ \tilde{F}_i(q): 1 \leq 1 \leq 6 \}$:
\begin{align*}
g_1(q)&:= \tilde{F}_1(q)+\tilde{F}_5(q)=q^{-7/80}(2+12q+30q^2+\cdots) \in q^{-7/80} \mathbb{C}[[q]], \\
g_2(q)&:=\tilde{F}_1(q)-\tilde{F}_5(q)=q^{33/80}(6+18q+54q^2+\cdots) \in q^{33/80}\mathbb{C}[[q]], \\
g_3(q)&:=\tilde{F}_4(q)+\tilde{F}_6(q)=q^{17/80}(4+6q+30q^2+\cdots) \in q^{17/80} \mathbb{C}[[q]], \\
g_4(q)&:=\tilde{F}_4(q)-\tilde{F}_6(q)=q^{57/80}(6+16q+42q^2+\cdots)  \in q^{57/80}\mathbb{C}[[q]],  \\
g_5(q)&:=\tilde{F}_2(q)=q^{1/40}(1+6q+15q^2+\cdots) \in q^{1/40} \mathbb{C}[[q]], \\ g_6(q)&:=\tilde{F}_3(q)=q^{9/40}(3+11q+30q^2+\cdots) \in q^{9/40}\mathbb{C}[[q]].
\end{align*}
Observe that now the leading $q$-series powers are non-congruent 
modulo $\mathbb{Z}$.
\end{rem}

\section{Proof of Theorem \ref{main2}}\label{sec-identities}

Recall the $q$-binomial theorem \cite[Theorem 1.3.1]{Berndt-book}:
\begin{align}\label{q-binomial}
\sum_{n=0}^\infty \frac{(a;q)_n}{(q;q)_n}z^n=\frac{(az;q)_\infty}{(z;q)_\infty}, \quad |z|<1.
\end{align}
As important corollaries of this theorem, Euler's $q$-exponential identities state that \cite[Corollary 1.3.2]{Berndt-book}
\begin{align}\label{Euler}
\sum_{n=0}^\infty \frac{z^n}{(q;q)_n}=\frac{1}{(z;q)_\infty}, \quad \sum_{n=0}^\infty \frac{q^{\binom{n}{2}} z^n}{(q;q)_n}=(-z;q)_\infty, \quad |z|<1.
\end{align}
The Jacobi triple product identity \cite[Theorem 1.3.3]{Berndt-book} is
\begin{align}\label{Jacobi}
(q,z,q/z;q)_\infty=\sum_{n=-\infty}^\infty (-1)^nq^{\binom{n}{2}}z^n.
\end{align}
It gives product representations for two important unary Jacobi theta functions:
\begin{align}
&\theta_2(\tau):=\sum_{n\in \mathbb{Z}}q^{(n+1/2)^2}=2q^{1/4}\frac{J_4^2}{J_2}, \label{theta2-prod} \\
&\theta_3(\tau):=\sum_{n\in \mathbb{Z}}q^{n^2}=\frac{J_2^5}{J_1^2J_4^2}. \label{theta3-prod}
\end{align}

We will use the constant term method. For any series $f(z)=\sum_{n \in \mathbb{Z}}a(n)z^n$, we define the operator
$$\mathrm{CT}[f(z)]=a(0),$$
which extracts the constant term of $f(z)$. Obviously, for any complex number $\alpha$ and integer $\beta$ with $\alpha \beta\neq 0$, we have
\begin{align}\label{constant-id}
\mathrm{CT}[f(\alpha z^\beta)]=\mathrm{CT}[f(z)].
\end{align}
\begin{proof}[Proof of Theorem \ref{main2}]
After replacing $q$ by $q^2$ in \eqref{eq-chi}, we have
\begin{align}
\chi_0(x_1,x_2,x_3;q^2)=\sum_{i,j,k\geq 0} \frac{q^{2i^2+2j^2+k^2-2ij-2jk}x_1^ix_2^jx_3^k}{(q^2;q^2)_i(q^2;q^2)_j(q^2;q^2)_k}. \label{F-start}
\end{align}
We have by \eqref{Euler} and \eqref{Jacobi} that
\begin{align}
\chi_0(x_1,x_2,x_3;q^2)&=\sum_{i,j\geq 0}\frac{q^{2i^2+2j^2-2ij}x_1^ix_2^j}{(q^2;q^2)_i(q^2;q^2)_j}\sum_{k\geq 0} \frac{q^{k^2-k}\cdot q^{(1-2j)k}x_3^k}{(q^2;q^2)_k} \nonumber \\
&=\sum_{i,j\geq 0} \frac{q^{2i^2+2j^2-2ij}x_1^ix_2^j}{(q^2;q^2)_i(q^2;q^2)_j} (-x_3q^{1-2j};q^2)_\infty. \label{F-exp}
\end{align}

(1) We have by \eqref{F-exp} that
\begin{align}
&\chi_0(1,1,1;q^2)=\sum_{i\geq 0} \frac{q^{2i^2+2j^2-2ij}}{(q^2;q^2)_i(q^2;q^2)_j}(-q^{1-2j};q^2)_\infty \nonumber \\
&=(-q;q^2)_\infty \sum_{i,j\geq 0} \frac{q^{2i^2+j^2-2ij}(-q;q^2)_j}{(q^2;q^2)_i(q^2;q^2)_j} \nonumber \\
&=(-q;q^2)_\infty \mathrm{CT}\left[ \sum_{i\geq 0} \frac{q^{i^2}z^i}{(q^2;q^2)_i}\sum_{j\geq 0} \frac{(-q;q^2)_jz^{-j}}{(q^2;q^2)_j} \sum_{k=-\infty}^\infty z^{-k}q^{k^2}\right] \nonumber \\
&=(-q;q^2)_\infty \mathrm{CT}\left[ \frac{(-qz,-q/z,-qz,-q/z,q^2;q^2)_\infty}{(1/z;q^2)_\infty} \right] \quad (\text{by \eqref{Euler} and \eqref{Jacobi})}\nonumber \\
&=\frac{(-q;q^2)_\infty }{(q^2;q^2)_\infty} \mathrm{CT}\left[ \frac{(-qz,-q/z,q^2;q^2)_\infty^2}{(1/z;q^2)_\infty}\right]  \nonumber \\
&=\frac{(-q;q^2)_\infty }{(q^2;q^2)_\infty} \mathrm{CT}\left[ \sum_{n=0}^\infty \frac{z^{-n}}{(q^2;q^2)_n} \sum_{i=-\infty}^\infty z^{i}q^{i^2} \sum_{j=-\infty}^\infty z^{j}q^{j^2} \right]  \quad \text{(by \eqref{Euler} and \eqref{Jacobi})} \nonumber \\
&=\frac{(-q;q^2)_\infty }{(q^2;q^2)_\infty} \sum_{n=0}^\infty \frac{1}{(q^2;q^2)_n} \sum_{i+j=n}q^{i^2+j^2} \nonumber \\
&=\frac{(-q;q^2)_\infty }{(q^2;q^2)_\infty} \sum_{n=0}^\infty \frac{q^{n^2}}{(q^2;q^2)_n} \sum_{i=-\infty}^\infty q^{2i^2-2ni} \nonumber \\
&=\frac{(-q;q^2)_\infty }{(q^2;q^2)_\infty}(S_0(q)+S_1(q)). \label{1-F-split}
\end{align}
Here $S_0(q)$ and $S_1(q)$ correspond to the sums with $n$ being even and odd, respectively.

We have
\begin{align}
S_0(q)&=\sum_{n=0}^\infty \frac{q^{2n^2}}{(q^2;q^2)_{2n}}\sum_{i=-\infty}^\infty q^{2(i-n)^2}=\sum_{n=0}^\infty \frac{q^{2n^2}}{(q^2;q^2)_{2n}}\sum_{i=-\infty}^\infty q^{2i^2} \nonumber \\
&=\frac{J_4^6J_{40}}{J_2^3J_8^2J_{8,40}}.  \quad \text{(by \eqref{S79} and \eqref{theta3-prod})} \label{1-S0-result}
\end{align}
Similarly, we have
\begin{align}
S_1(q)&=\sum_{n=0}^\infty \frac{q^{4n^2+4n+1}}{(q^2;q^2)_{2n+1}} \sum_{i=-\infty}^\infty q^{2i^2-4in-2i} =\sum_{n=0}^\infty \frac{q^{2n^2+2n+1}}{(q^2;q^2)_{2n+1}} \sum_{i=-\infty}^\infty q^{2(i-n)^2-2(i-n)} \nonumber \\
&=\sum_{n=0}^\infty \frac{q^{2n^2+2n+1}}{(q^2;q^2)_{2n+1}} \sum_{i=-\infty}^\infty q^{2i^2-2i} =2q\frac{J_8^2J_{6,20}J_{8,40}}{J_2J_4J_{40}}.  \quad \text{(by \eqref{S94} and \eqref{theta2-prod})}\label{1-S1-result}
\end{align}
Substituting \eqref{1-S0-result} and \eqref{1-S1-result} into \eqref{1-F-split}, we obtain \eqref{sum-id-1}.

(2) We have by \eqref{F-exp} that
\begin{align}
&\chi_0(1,1,q;q^2)=\sum_{i,j\geq 0} \frac{q^{2i^2+2j^2-2ij}}{(q^2;q^2)_i(q^2;q^2)_j}(-q^{2-2j};q^2)_\infty \nonumber \\
&=(-q^2;q^2)_\infty \sum_{i,j\geq 0} \frac{q^{2i^2+j^2+j-2ij}}{(q^2;q^2)_i(q^2;q^2)_j}(-1;q^2)_j \nonumber \\
&=(-q^2;q^2)_\infty \mathrm{CT}\left[ \sum_{i\geq 0} \frac{q^{i^2}z^i}{(q^2;q^2)_i} \sum_{j\geq 0} \frac{q^jz^{-j}(-1;q^2)_j}{(q^2;q^2)_j} \sum_{k=-\infty}^\infty z^{-k}q^{k^2}\right] \nonumber \\
&=(-q^2;q^2)_\infty \mathrm{CT} \left[ \frac{(-qz,-q/z,-q/z,-qz,q^2;q^2)_\infty}{(q/z;q^2)_\infty} \right] \quad (\text{by \eqref{Euler} and \eqref{Jacobi})}\label{2-F-int} \\
&=\frac{(-q^2;q^2)_\infty}{(q^2;q^2)_\infty} \mathrm{CT}\left[ \sum_{n=0}^\infty \frac{q^n}{(q^2;q^2)_n} \sum_{i=-\infty}^\infty z^{i}q^{i^2} \sum_{j=-\infty}^\infty z^jq^{j^2} \right] \quad \text{(by \eqref{Euler} and \eqref{Jacobi})} \nonumber \\
&=\frac{(-q^2;q^2)_\infty}{(q^2;q^2)_\infty} \sum_{n=0}^\infty \frac{q^n}{(q^2;q^2)_n} \sum_{i+j=n} q^{i^2+j^2}  \nonumber \\
&=\frac{(-q^2;q^2)_\infty}{(q^2;q^2)_\infty} \sum_{n=0}^\infty \frac{q^{n^2+n}}{(q^2;q^2)_n} \sum_{i=-\infty}^\infty q^{2i^2-2ni} \nonumber \\
&=\frac{(-q^2;q^2)_\infty}{(q^2;q^2)_\infty}(S_0(q)+S_1(q)). \label{2-F-split}
\end{align}
Here $S_0(q)$ and $S_1(q)$ correspond to the sums with $n$ being even and odd, respectively.

We have
\begin{align}
S_0(q)&=\sum_{n=0}^\infty \frac{q^{4n^2+2n}}{(q^2;q^2)_{2n}}\sum_{i=-\infty}^\infty q^{2i^2-4ni} =\sum_{n=0}^\infty \frac{q^{2n^2+2n}}{(q^2;q^2)_{2n}}\sum_{i=-\infty}^\infty q^{2(i-n)^2} \nonumber \\
&=\sum_{n=0}^\infty \frac{q^{2n^2+2n}}{(q^2;q^2)_{2n}}\sum_{i=-\infty}^\infty q^{2i^2} =\frac{J_4^5J_{2,20}J_{16,40}}{J_2^3J_8^2J_{40}}.  \quad \text{(by \eqref{S99} and \eqref{theta3-prod})}\label{2-S0-result}
\end{align}
Similarly,
\begin{align}
S_1(q)&=\sum_{n=0}^\infty \frac{q^{4n^2+6n+2}}{(q^2;q^2)_{2n+1}} \sum_{i=-\infty}^\infty q^{2i^2-4ni-2i} =\sum_{n=0}^\infty \frac{q^{2n^2+4n+2}}{(q^2;q^2)_{2n+1}} \sum_{i=-\infty}^\infty q^{2(n-i)^2+2(n-i)} \nonumber \\
&=\sum_{n=0}^\infty \frac{q^{2n^2+4n+2}}{(q^2;q^2)_{2n+1}} \sum_{i=-\infty}^\infty q^{2i^2+2i} =2q^2\frac{J_8^2J_{8,20}J_{4,40}}{J_2J_4J_{40}}. \quad \text{(by \eqref{Rogers-1} and \eqref{theta2-prod})} \label{2-S1-result}
\end{align}
Substituting \eqref{2-S0-result} and \eqref{2-S1-result} into \eqref{2-F-split}, we obtain \eqref{sum-id-2}.

(3) We have by \eqref{F-exp} that
\begin{align}
&\chi_0(q^2,q^{-2},q;q^2)=\sum_{i,j\geq 0} \frac{q^{2i^2+2j^2-2ij+2i-2j}}{(q^2;q^2)_i(q^2;q^2)_j}(-q^{2-2j};q^2)_\infty \nonumber \\
&=(-q^2;q^2)_\infty \sum_{i,j\geq 0} \frac{q^{2i^2+j^2-j-2ij+2i}}{(q^2;q^2)_i(q^2;q^2)_j}(-1;q^2)_j \nonumber \\
&=(-q^2;q^2)_\infty \mathrm{CT}\left[ \sum_{i\geq 0} \frac{q^{i^2+2i}z^i}{(q^2;q^2)_i} \sum_{j\geq 0} \frac{q^{-j}z^{-j}(-1;q^2)_j}{(q^2;q^2)_j} \sum_{k=-\infty}^\infty z^{-k}q^{k^2} \right] \nonumber \\
&=(-q^2;q^2)_\infty \mathrm{CT}\left[ \frac{(-q^3z,-1/(qz);q^2)_\infty (-qz,-q/z,q^2;q^2)_\infty}{(1/(qz);q^2)_\infty} \right] 
 \quad (\text{by \eqref{Euler} and \eqref{Jacobi})}\label{3-F-int} \\
&=\frac{(-q^2;q^2)_\infty}{(q^2;q^2)_\infty} \mathrm{CT}\left[ \frac{(-q^3z,-1/(qz),q^2;q^2)_\infty (-qz,-q/z,q^2;q^2)_\infty}{(1/(qz);q^2)_\infty} \right] \nonumber \\
&=\frac{(-q^2;q^2)_\infty}{(q^2;q^2)_\infty} \mathrm{CT}\left[ \sum_{n=0}^\infty \frac{q^{-n}z^{-n}}{(q^2;q^2)_n} \sum_{i=-\infty}^\infty q^{i^2+2i}z^i \sum_{j=-\infty}^\infty q^{j^2}z^{j} \right] \quad (\text{by \eqref{Euler} and \eqref{Jacobi})}\nonumber \\
&=\frac{(-q^2;q^2)_\infty}{(q^2;q^2)_\infty} \sum_{n=0}^\infty \frac{q^{-n}}{(q^2;q^2)_n} \sum_{i+j=n} q^{i^2+2i+j^2} \nonumber \\
&=\frac{(-q^2;q^2)_\infty}{(q^2;q^2)_\infty} \sum_{n=0}^\infty \frac{q^{n^2-n}}{(q^2;q^2)_n} \sum_{i=-\infty}^\infty q^{2i^2-2ni+2i} \nonumber \\
&=\frac{(-q^2;q^2)_\infty}{(q^2;q^2)_\infty} (S_0(q)+S_1(q)). \label{3-F-split}
\end{align}
Here $S_0(q)$ and $S_1(q)$ correspond to the sums with $n$ being even and odd, respectively.

We have
\begin{align}
S_0(q)&=\sum_{n=0}^\infty \frac{q^{4n^2-2n}}{(q^2;q^2)_{2n}} \sum_{i=-\infty}^\infty q^{2i^2-4ni+2i} =\sum_{n=0}^\infty \frac{q^{2n^2}}{(q^2;q^2)_{2n}} \sum_{i=-\infty}^\infty q^{2(i-n)^2+2(i-n)} \nonumber \\
&=\sum_{n=0}^\infty \frac{q^{2n^2}}{(q^2;q^2)_{2n}} \sum_{i=-\infty}^\infty q^{2i^2+2i}=2\frac{J_8^2J_{40}}{J_2J_{8,40}}. \quad \text{(by \eqref{S79} and \eqref{theta2-prod})}\label{3-S0}
\end{align}
Similarly,
\begin{align}
S_1(q)&=\sum_{n=0}^\infty \frac{q^{4n^2+2n}}{(q^2;q^2)_{2n+1}} \sum_{i=-\infty}^\infty q^{2i^2-4in}=\sum_{n=0}^\infty \frac{q^{2n^2+2n}}{(q^2;q^2)_{2n+1}} \sum_{i=-\infty}^\infty q^{2(i-n)^2} \nonumber \\
&=\sum_{n=0}^\infty \frac{q^{2n^2+2n}}{(q^2;q^2)_{2n+1}} \sum_{i=-\infty}^\infty q^{2i^2} =\frac{J_4^5J_{6,20}J_{8,40}}{J_2^3J_8^2J_{40}}.  \quad \text{(by \eqref{S94} and \eqref{theta3-prod})}\label{3-S1}
\end{align}
Substituting \eqref{3-S0} and \eqref{3-S1} into \eqref{3-F-split}, we obtain \eqref{sum-id-3}.

(4) We have by \eqref{F-exp} that
\begin{align}
&\chi_0(q^{-2},q^2,1;q^2)=\sum_{i,j\geq 0} \frac{q^{2i^2+2j^2-2ij-2i+2j}}{(q^2;q^2)_i(q^2;q^2)_j}(-q^{1-2j};q^2)_\infty \nonumber \\
&=(-q;q^2)_\infty \sum_{i,j\geq 0} \frac{q^{2i^2+j^2-2ij-2i+2j}}{(q^2;q^2)_i(q^2;q^2)_j}(-q;q^2)_j \nonumber \\
&=(-q;q^2)_\infty \mathrm{CT}\left[ \sum_{i\geq 0} \frac{q^{i^2-2i}z^i}{(q^2;q^2)_i} \sum_{j\geq 0} \frac{q^{2j}z^{-j}(-q;q^2)_j}{(q^2;q^2)_j} \sum_{k=-\infty}^\infty z^{-k}q^{k^2} \right] \nonumber \\
&=(-q;q^2)_\infty \mathrm{CT}\left[ \frac{(-z/q,-q^3/z,-qz,-q/z,q^2;q^2)_\infty}{(q^2/z;q^2)_\infty} \right] \quad (\text{by \eqref{Euler} and \eqref{Jacobi})}\nonumber \\
&=\frac{(-q;q^2)_\infty}{(q^2;q^2)_\infty} \mathrm{CT}\left[ \frac{(-z/q,-q^3/z,q^2;q^2)_\infty (-qz,-q/z,q^2;q^2)_\infty}{(q^2/z;q^2)_\infty} \right]\nonumber \\
&=\frac{(-q;q^2)_\infty}{(q^2;q^2)_\infty} \mathrm{CT}\left[ \sum_{n=0}^\infty \frac{q^{2n}z^{-n}}{(q^2;q^2)_n} \sum_{i=-\infty}^\infty z^{i} q^{i^2-2i} \sum_{j=-\infty}^\infty  z^jq^{j^2} \right] \quad \text{(by \eqref{Euler} and \eqref{Jacobi})} \nonumber \\
&=\frac{(-q;q^2)_\infty}{(q^2;q^2)_\infty}  \sum_{n=0}^\infty \frac{q^{2n}}{(q^2;q^2)_n} \sum_{i+j=n}q^{i^2+j^2-2i} \nonumber \\
&=\frac{(-q;q^2)_\infty}{(q^2;q^2)_\infty}  \sum_{n=0}^\infty \frac{q^{n^2+2n}}{(q^2;q^2)_n} \sum_{i=-\infty}^\infty q^{2i^2-2ni-2i} \nonumber \\
&=\frac{(-q;q^2)_\infty}{(q^2;q^2)_\infty} (S_0(q)+S_1(q)). \label{4-F-split}
\end{align}
Here $S_0(q)$ and $S_1(q)$ correspond to the sums with $n$ being even and odd, respectively.

We have
\begin{align}
S_0(q)&=\sum_{n=0}^\infty \frac{q^{4n^2+4n}}{(q^2;q^2)_{2n}}\sum_{i=-\infty}^\infty q^{2i^2-4in-2i} =\sum_{n=0}^\infty \frac{q^{2n^2+2n}}{(q^2;q^2)_{2n}} \sum_{i=-\infty}^\infty q^{2(n-i)^2+2(n-i)} \nonumber \\
&=\sum_{n=0}^\infty \frac{q^{2n^2+2n}}{(q^2;q^2)_{2n}} \sum_{i=-\infty}^\infty q^{2i^2+2i} =2\frac{J_8^2J_{2,20}J_{16,40}}{J_2J_4J_{40}}. \quad \text{(by \eqref{S99} and \eqref{theta2-prod})}\label{4-S0-result}
\end{align}
Similarly,
\begin{align}
S_1(q)&=\sum_{n=0}^\infty \frac{q^{4n^2+8n+3}}{(q^2;q^2)_{2n+1}}\sum_{i=-\infty}^\infty q^{2i^2-4ni-4i} =\sum_{n=0}^\infty \frac{q^{2n^2+4n+3}}{(q^2;q^2)_{2n+1}}\sum_{i=-\infty}^\infty q^{2(n-i)^2+4(n-i)} \nonumber \\
&=q\sum_{n=0}^\infty \frac{q^{2n^2+4n}}{(q^2;q^2)_{2n+1}}\sum_{i=-\infty}^\infty q^{2(i+1)^2} =q \frac{J_4^5J_{8,20}J_{4,40}}{J_2^3J_8^2J_{40}}.  \quad \text{(by \eqref{Rogers-1} and \eqref{theta3-prod})}\label{4-S1-result}
\end{align}
Substituting \eqref{4-S0-result} and \eqref{4-S1-result} into \eqref{4-F-split}, we obtain \eqref{sum-id-4}.

(5) We have by \eqref{F-exp} that
\begin{align}
&\chi_0(q^{-2},q^2,q^{-1};q^2)=\sum_{i,,j\geq 0} \frac{q^{2i^2+2j^2-2ij-2i+2j}}{(q^2;q^2)_i(q^2;q^2)_j}(-q^{-2j};q^2)_\infty \nonumber \\
&=(-1;q^2)_\infty \sum_{i,j\geq 0} \frac{q^{2i^2+j^2+j-2ij-2i}}{(q^2;q^2)_i(q^2;q^2)_j}(-q^2;q^2)_j \nonumber \\
&=2(-q^2;q^2)_\infty \mathrm{CT}\left[ \sum_{i\geq 0} \frac{q^{i^2-2i}z^i}{(q^2;q^2)_i} \sum_{j\geq 0} \frac{(-q^2;q^2)_jq^jz^{-j}}{(q^2;q^2)_j} \sum_{k=-\infty}^\infty z^{-k}q^{k^2} \right] \nonumber \\
&=2(-q^2;q^2)_\infty \mathrm{CT}\left[ \frac{(-q^3/z,-z/q,-qz,-q/z,q^2;q^2)_\infty}{(q/z;q^2)_\infty} \right]. \label{5-F-int}
\end{align}
Now if we replace $z$ by $q^2z$ in \eqref{5-F-int}, which does not change the result by \eqref{constant-id}, and then compare with \eqref{3-F-int}, after replacing $q^2$ by $ q$, we obtain \eqref{Nahm-relation-1}. 
In view of \eqref{sum-id-3}, we obtain \eqref{sum-id-5}.

(6) We have by \eqref{F-exp} that
\begin{align}
&\chi_0(q^{-4},q^4,q^{-1};q^2)=\sum_{i,j\geq 0} \frac{q^{2i^2+2j^2-2ij-4i+4j}}{(q^2;q^2)_i(q^2;q^2)_j} (-q^{-2j};q^2)_\infty \nonumber \\
&=(-1;q^2)_\infty \sum_{i,j\geq 0} \frac{q^{2i^2+j^2-2ij-4i+3j}}{(q^2;q^2)_i(q^2;q^2)_j}(-q^2;q^2)_j \nonumber \\
&=2(-q^2;q^2)_\infty \mathrm{CT}\left[ \sum_{i\geq 0} \frac{q^{i^2-4i}z^i}{(q^2;q^2)_i} \sum_{j\geq 0} \frac{q^{3j}z^{-j}(-q^2;q^2)_j}{(q^2;q^2)_j} \sum_{k=-\infty}^\infty z^{-k}q^{k^2} \right] \nonumber \\
&=2(-q^2;q^2)_\infty \mathrm{CT}\left[  \frac{(-q^{-3}z,-q^5/z,-qz,-q/z,q^2;q^2)_\infty}{(q^3/z;q^2)_\infty} \right] \quad \text{(by \eqref{Euler} and \eqref{Jacobi})} \nonumber \\
&=2(-q^2;q^2)_\infty \mathrm{CT}\left[  \frac{(-q^{-1}z,-q^3/z,-q^3z,-1/(qz),q^2;q^2)_\infty}{(q/z;q^2)_\infty} \right] \quad \text{(replace $z$ by $q^2z$)}  \nonumber \\
&=2(-q^2;q^2)_\infty \mathrm{CT}\left[  \frac{(1+q^{-1}z)(1+q^{-1}z^{-1})}{(1+qz)(1+qz^{-1})}\cdot \frac{(-qz,-q/z,-qz,-q/z,q^2;q^2)_\infty}{(q/z;q^2)_\infty} \right] \nonumber \\
&=2q^{-2}(-q^2;q^2)_\infty \mathrm{CT}\left[ \frac{(-qz,-q/z,-qz,-q/z,q^2;q^2)_\infty}{(q/z;q^2)_\infty} \right] . \label{6-F-int}
\end{align}
Note that \eqref{6-F-int} and \eqref{2-F-int} differ only by the factor $2q^{-2}$. After replacing $q^2$ by $q$, this proves \eqref{Nahm-relation-2}.
In view of \eqref{sum-id-2}, we obtain \eqref{sum-id-6}.
\end{proof}

\section{Proofs of Theorem \ref{main}}\label{sec-proofs}
In this section, we provide two different proofs for Theorem \ref{main}. In the first proof, we treat the functions $\tilde{F}_i(\tau)$ ($1\leq i\leq 6$) together by viewing them as vector-valued modular form. In the second proof, we treat these identities one by one.

\subsection{Theta functions} In this subsection, we shall make some preparations for the proofs. 
Recall the full modular group 
$$\mathrm{SL}(2,\mathbb{Z})=\left\{\begin{pmatrix} a & b \\ c & d \end{pmatrix}: a,b,c,d\in \mathbb{Z}, ad-bc=1\right\}.$$
This group is also conveniently denoted as $\Gamma(1)$. It is generated by the matrices
$$S=\begin{pmatrix}
    0 & -1 \\ 1 & 0
\end{pmatrix}, \quad T=\begin{pmatrix}
    1 & 1 \\ 0 & 1
\end{pmatrix}.$$
For any congruence subgroup $G$ of $\mathrm{SL}(2,\mathbb{Z})$ and Dirichlet character $\chi$, we use $M_k(G, \chi)$ (resp.\ $S_k(G,\chi)$) to denote the space of modular forms (resp.\ cusp forms) on $G$ with weight $k$ and multiplier $\chi$. When $\chi$ is trivial, we omit it and write the space as $M_k(G)$ (resp.\ $S_k(G)$). Besides $\Gamma(1)$ itself, we will mainly work with the congruence subgroups
\begin{align}
&\Gamma_1(N):=\left\{\begin{pmatrix} a & b \\ c & d  \end{pmatrix} \in \mathrm{SL}(2,\mathbb{Z}), \quad \begin{pmatrix} a & b \\ c & d \end{pmatrix} \equiv \begin{pmatrix} 1 & * \\ 0 & 1 \end{pmatrix} \pmod{N}\right\}, \\
&\Gamma_0(N):=\left\{\begin{pmatrix} a & b \\ c & d  \end{pmatrix} \in \mathrm{SL}(2,\mathbb{Z}), \quad \begin{pmatrix} a & b \\ c & d \end{pmatrix} \equiv \begin{pmatrix} * & * \\ 0 & * \end{pmatrix} \pmod{N}\right\}.
\end{align}

It is well-known that the Weber modular functions $\mathfrak{f}(\tau), \mathfrak{f}_1(\tau)$ and $\mathfrak{f}_2(\tau)$ defined in \eqref{Weber-defn} transform as a vector-valued modular form of rank three under $\Gamma(1)$:
$$\mathfrak{f}(-1/\tau) = \mathfrak{f}(\tau), \ \ \ \mathfrak{f}_2(-1/\tau) = \frac{1}{\sqrt{2}}  \mathfrak{f}_1(\tau), \ \ \mathfrak{f}_1(-1/\tau)=\sqrt{2} \mathfrak{f}_2(\tau)$$
$$\mathfrak{f}(\tau+1) =e^{-\pi i /24}  \mathfrak{f}_1(\tau), \ \ \mathfrak{f}_1(\tau+1) = e^{-\pi i/24} \mathfrak{f}(\tau),  \ \ \mathfrak{f}_2(\tau+1) = e^{\pi i/12} \mathfrak{f}_2(\tau).$$
We also require the Dedekind $\eta$-functions $\eta(\tau)$ and its transformation properties:
$$\eta(-1/\tau)=\sqrt{-i \tau} \eta(\tau), \ \ \eta(\tau+1)=e^{\pi i  /12} \eta(\tau).$$

Recall the series $(\partial \Theta)_{j,k}(\tau)$ and $(\partial G)_{j,k}(\tau)$ defined in \eqref{partial-Theta-defn} and \eqref{partial-G-defn}.  The following lemma summarizes some useful properties and especially some transformation formulas for them.
\begin{lemma}
\begin{enumerate}[(1)]
\item For any $k>0$ and $j$ we have 
\begin{align}
&(\partial \Theta)_{0,k}(\tau)=0, \quad (\partial \Theta)_{k,k}(\tau)=0, \label{Theta-zero} \\
&(\partial \Theta)_{j,\frac{k}{2}}(\tau)=-(\partial \Theta)_{k-j,\frac{k}{2}}(\tau), \ \ \ (\partial G)_{j,\frac{k}{2}}(\tau)=(\partial G)_{k-j,\frac{k}{2}}(\tau), \label{symmetry} \\
&(\partial \Theta)_{j,k}(\tau)=\frac{1}{2}(\partial \Theta)_{2j,4k}(\tau)+\frac{1}{2}(\partial \Theta)_{2j+4k,4k}(\tau), \label{Theta-dissection} \\
&(\partial G)_{j,k}(\tau)=\frac{1}{2}(\partial \Theta)_{2j,4k}(\tau)-\frac{1}{2}(\partial \Theta)_{2j+4k,4k}(\tau). \label{G-dissection}
\end{align}
\item For $k\in \mathbb{N}+\frac{1}{2}$ and $j\in \mathbb{N}$ we have
\begin{align}
& (\partial \Theta)_{j,k}(\tau+1)=e^{i \pi j^2/2k} (\partial G)_{j,k}(\tau),  \label{Theta-T}\\
&
 (\partial G)_{j,k}(\tau+1)=e^{i \pi j^2/2k}(\partial \Theta)_{j,k}(\tau), \label{G-T}  \\
& (\partial \Theta)_{j,k}(\tau+2)=e^{i \pi j^2/k}  (\partial \Theta)_{j,k}(\tau), \label{Theta-T-twice} \\
&
 (\partial G)_{j,k}(\tau+2)=e^{i \pi j^2/k}(\partial G)_{j,k}(\tau). \label{Theta-G-twice}
\end{align}
For $k\in \mathbb{N}+\frac{1}{2}$ and $j\in \mathbb{N}+\frac{1}{2}$ we have
\begin{align}
& (\partial \Theta)_{j,k}(\tau+1)=e^{i \pi j^2/2k}  (\partial \Theta)_{j,k}(\tau), \\
& (\partial G)_{j,k}(\tau+1)=e^{i \pi j^2/2k}  (\partial G)_{j,k}(\tau).
\end{align}
\item For $k\in \frac{1}{2}\mathbb{N}$ and $j\in \mathbb{N}$ we have
\begin{align}
 (\partial \Theta)_{j,k}(-1/\tau)&=(-\tau) \sqrt{-i \tau/2k} \sum_{j'=1}^{2k-1} e^{i \pi j j'/k} (\partial \Theta)_{j',k} (\tau), \label{Theta-mod} \\
  (\partial G)_{j,k}(-1/{\tau})&=(-\tau)\sqrt{-i\tau/2k} \sum_{j'=1}^{2k} e^{i\pi j(2j'-1)/k}(\partial \Theta)_{\frac{2j'-1}{2},k}(\tau). \label{theta-g-T}
\end{align}
\item For $k\in \frac{1}{2}\mathbb{N}$ and $j\in \mathbb{N}+\frac{1}{2}$ we have
\begin{align}
(\partial \Theta)_{j,k}(-1/\tau)&=(-\tau) \sqrt{-i \tau/2k} \sum_{j'=1}^{2k-1} e^{i \pi j j'/k} (\partial G)_{j',k} (\tau), \label{theta-g}  \\
  (\partial G)_{j,k}(-1/{\tau})&=(-\tau)\sqrt{-i\tau/2k} \sum_{j'=1}^{2k} e^{i\pi j(2j'-1)/k}(\partial G)_{\frac{2j'-1}{2},k}(\tau). \label{theta-g-T-second}
\end{align}
\end{enumerate}
\end{lemma}

\begin{proof}
Replacing $n$ by $-n$ in \eqref{partial-Theta-defn}, we deduce that $(\partial \Theta)_{0,k}(\tau)=0$. Replacing $n$ by $-n-1$, we obtain $(\partial \Theta)_{k,k}(\tau)=0$. This proves \eqref{Theta-zero}. In the same way we can prove \eqref{symmetry}.

Splitting the sum according to the parity of $n$, we have
\begin{align*}
&(\partial \Theta)_{j,k}(\tau)=\sum_{n\in \mathbb{Z}}(4kn+j)q^{(4kn+j)^2/4k}+\sum_{n\in \mathbb{Z}} (4kn+2k+j)q^{(4kn+2k+j)^2/(4k)} \nonumber \\
&=\frac{1}{2}\sum_{n\in \mathbb{Z}}(8kn+2j)q^{(8kn+2j)^2/(16k)}+\frac{1}{2}\sum_{n\in \mathbb{Z}}(8kn+4k+2j)q^{(8kn+4k+2j)^2/(16k)} \nonumber \\
&=\frac{1}{2}(\partial \Theta)_{2j,4k}(\tau)+\frac{1}{2}(\partial \Theta)_{2j+4k,4k}(\tau).
\end{align*}
This proves \eqref{Theta-dissection}. Similarly, we can prove \eqref{G-dissection} and hence finish the proof of part (1).

Part (2) follows from definition. 

It remains to prove parts (3) and (4). First,  the formula \eqref{Theta-mod} follows from \cite[Eq.\ (2.4)]{Shimura} and the fact $(\partial \Theta)_{0,k}(\tau)=0$ (see \eqref{Theta-zero}). It also appears as \cite[Eq.\ (12.8)]{AM}.

Now assume that $k\in \frac{1}{2}\mathbb{N}$ and $j\in \mathbb{N}+\frac{1}{2}$. Replacing $\tau$ by $-1/\tau$ in \eqref{Theta-dissection} and using \eqref{Theta-mod}, we deduce that
\begin{align*}
  &  (\partial \Theta)_{j,k}(-1/\tau)=\frac{1}{2}(-\tau)\sqrt{\frac{-i\tau}{8k}}\sum_{j'=1}^{8k-1}e^{i\pi jj'/2k}(\partial \Theta)_{j',4k}(\tau) \nonumber \\
&\quad    +\frac{1}{2}(-\tau)\sqrt{\frac{-i\tau}{8k}}\sum_{j'=1}^{8k-1}e^{i\pi (2k+j)j'/2k}(\partial \Theta)_{j',4k}(\tau) \nonumber \\
&=-\frac{1}{4}\tau\sqrt{\frac{-i\tau}{2k}}\sum_{j'=1}^{8k-1}(1+e^{i\pi j'})e^{i\pi jj'/2k}(\partial \Theta)_{j',4k}(\tau) \nonumber \\
&=-\frac{1}{2}\tau \sqrt{\frac{-i\tau}{2k}}\sum_{j'=1}^{4k-1}e^{i\pi jj'/k}(\partial \Theta)_{2j',4k}(\tau) \\
&=-\frac{1}{2}\tau \sqrt{\frac{-i\tau}{2k}} \sum_{j'=1}^{2k-1} \left(e^{i\pi jj'/k}(\partial \Theta)_{2j',4k}(\tau)+e^{i\pi j(j'+2k)/k})(\partial \Theta)_{2j'+4k,4k}(\tau)\right)\nonumber \\
&=-\frac{1}{2}\tau \sqrt{\frac{-i\tau}{2k}} \sum_{j'=1}^{2k-1} e^{i\pi jj'/k}\Big((\partial \Theta)_{2j',4k}(\tau)-(\partial \Theta)_{2j'+4k,4k}(\tau)\Big)\nonumber \\
&=-\tau \sqrt{\frac{-i\tau}{2k}} \sum_{j'=1}^{2k-1} e^{i\pi jj'/k} (\partial G)_{j,k}(\tau).
\end{align*}
Here for the last third line we used the fact that $(\partial \Theta)_{4k,4k}(\tau)=0$ (see \eqref{Theta-zero}). This proves \eqref{theta-g}. 

For $k\in \frac{1}{2}\mathbb{N}$ and $j\in \frac{1}{2}\mathbb{N}$, replacing $\tau$ by $-1/\tau$ in \eqref{G-dissection} and using \eqref{Theta-mod}, we deduce that
\begin{align*}
   & (\partial G)_{j,k}(-1/\tau)=\frac{1}{2}(-\tau)\sqrt{\frac{-i\tau}{8k}}\sum_{j'=1}^{8k-1}(e^{i\pi jj'/2k}-e^{i\pi (j+2k)j'/2k})(\partial \Theta)_{j',4k}(\tau) \nonumber \\
&=(-\tau)\sqrt{\frac{-i\tau}{8k}}\sum_{j'=1}^{4k}e^{i\pi j(2j'-1)/2k}(\partial \Theta)_{2j'-1,4k}(\tau) \nonumber \\
&=(-\tau)\sqrt{\frac{-i\tau}{8k}} \sum_{j'=1}^{2k} \left(e^{i\pi j(2j'-1)/2k}(\partial \Theta)_{2j'-1,4k}(\tau)+e^{i\pi j(2j'+4k-1)/2k}(\partial \Theta)_{2j'+4k-1,4k}(\tau) \right) \nonumber \\
&=(-\tau)\sqrt{\frac{-i\tau}{8k}} \sum_{j'=1}^{2k} e^{i\pi j(2j'-1)/2k} \Big( (\partial \Theta)_{2j'-1,4k}(\tau)+e^{2i\pi j}(\partial \Theta)_{2j'+4k-1,4k}(\tau) \Big).
\end{align*}
Discussing according to  $j\in \mathbb{N}$ or  $j\in \mathbb{N}+\frac{1}{2}$ and using \eqref{Theta-dissection}--\eqref{G-dissection}, we obtain \eqref{theta-g-T} and \eqref{theta-g-T-second}, respectively.
\end{proof}

Combining these formulas give
\begin{prop} Let $k$ be a positive integer. Then the following functions:
\begin{align*}
&  \frac{\mathfrak{f}(\tau)^3}{\eta(\tau)^3} (\partial \Theta)_{i,\frac{2k+1}{2}}(\tau), \quad 1 \leq i \leq k, \\
&  \frac{\mathfrak{f}_1(\tau)^3}{\eta(\tau)^3} (\partial G)_{i,\frac{2k+1}{2}}(\tau), \quad 1 \leq i \leq k, \\
&  \frac{\mathfrak{f}_2(\tau)^3}{\eta(\tau)^3} (\partial  \Theta)_{\frac{2i-1}{2},\frac{2k+1}{2}}(\tau), \quad 1 \leq i \leq k,
\end{align*}
combine into a vector-valued modular function under $\Gamma(1)$.
 \end{prop}

\begin{proof} First notice that all components are modular functions on an appropriate congruence subgroup (see \cite{Bringmann} for instance). Thus it suffices to prove that the span of the functions (which are linearly independent) closes under $\tau \to \tau+1$ and $\tau \to -\frac{1}{\tau}$.
For $\tau \to \tau+1$ this is clear and follows immediately from the formulas above. Under $\tau \to -1/\tau$ we see using \eqref{Theta-mod} and \eqref{symmetry} that the span of $\frac{\mathfrak{f}^3(\tau)}{\eta^3(\tau)} (\partial \Theta)_{i,\frac{2k+1}{2}}(\tau)$ transforms to itself. Similarly, formulas $\mathfrak{f}_2(-1/\tau) = \frac{1}{\sqrt{2}}  \mathfrak{f}_1(\tau)$, \eqref{symmetry}, \eqref{theta-g-T} and \eqref{theta-g} show that the span of
$ \frac{\mathfrak{f}_2^3(\tau)}{\eta^3(\tau)} (\partial  \Theta)_{\frac{2i-1}{2},\frac{2k+1}{2}}(\tau)$ transforms under $\tau \to -1/\tau$ to the span of  $\frac{\mathfrak{f}_1^3(\tau)}{\eta^3(\tau)} (\partial G)_{i,\frac{2k+1}{2}}(\tau)$ and vice-versa. We proved the assertion.
\end{proof}

As a corollary we record
\begin{corollary} \label{gam}
(1) The vector-valued function 
$$\Big(\frac{\mathfrak{f}(\tau)^3}{\eta(\tau)^3} (\partial \Theta)_{1,\frac52}(\tau), \frac{\mathfrak{f}(\tau)^3}{\eta(\tau)^3} (\partial \Theta)_{2,\frac52}(\tau)\Big)$$
transforms as a vector-valued modular form (of weight zero) for $\Gamma_\theta=\langle S, T^2\rangle$, the sugbroup of $\Gamma(1)$ generated by $S$ and $T^2$.

(2) The vector-valued function
$$\biggl( \frac{\mathfrak{f}(\tau)^3}{\eta(\tau)^3} (\partial \Theta)_{1,\frac52}(\tau), \frac{\mathfrak{f}(\tau)^3}{\eta(\tau)^3} (\partial \Theta)_{2,\frac52}(\tau),\frac{\mathfrak{f}_1(\tau)^3}{\eta(\tau)^3} (\partial G)_{1,\frac52}(\tau), $$
$$\frac{\mathfrak{f}_1(\tau)^3}{\eta(\tau)^3} (\partial G)_{2,\frac52}(\tau),\frac{\mathfrak{f}_2(\tau)^3}{\eta(\tau)^3} (\partial \Theta)_{\frac12,\frac52}(\tau), \frac{\mathfrak{f}_2(\tau)^3}{\eta(\tau)^3} (\partial \Theta)_{\frac32,\frac52}(\tau)\biggr)$$
transforms as a vector-valued modular function for $\Gamma(1)$.
\end{corollary}

Part (2) of this corollary proves the second assertion in Theorem \ref{main}. Therefore, to finish the proof of Theorem \ref{main}, it suffices to prove the six identities.
Below we present two different proofs.

\subsection{First proof of Theorem \ref{main}}
We first recall several known facts. Recall the theta functions $\theta_2(\tau)$ and $\theta_3(\tau)$ defined in \eqref{theta2-prod} and \eqref{theta3-prod}. The pair
\begin{equation}
\label{sl2}
(W_1(\tau),W_2(\tau)):=\left(\dfrac{\theta_3(\tau)}{\eta(\tau)}, \dfrac{\theta_2(\tau)}{\eta(\tau)}\right)
\end{equation}
defines a two-dimensional vector-valued modular form  $\rho_1$ of weight zero whose $S$-matrix is given by
$$\begin{pmatrix} W_1(-1/\tau) \\ W_2(-1/\tau) \end{pmatrix} = \frac{1}{\sqrt{2}} \begin{pmatrix} 1 & 1 \\ 1 & -1 \end{pmatrix} \begin{pmatrix} W_1(\tau) \\ W_2(\tau) \end{pmatrix}  .$$
These two series are precisely the level one characters of standard $A_1^{(1)}$-modules.

Let also
$${\rm ch}_{3,5}^{r,s}(\tau):= \frac{1}{\eta(\tau)}\sum_{n \in \mathbb{Z}} \left(q^{\frac{(30n+5r-3s)^2}{60}}-q^{\frac{(30n+5r+3s)^2}{60}}\right).$$
This notation indicates that ${\rm ch}_{3,5}^{1,1}(\tau), {\rm ch}_{3,5}^{1,2}(\tau), {\rm ch}_{3,5}^{2,1}(\tau), {\rm ch}_{2,5}^{2,2}(\tau)$ are precisely characters of four irreducible $(3,5)$ Virasoro minimal models.
We need the following known identities obtained using the quintuple product identity \cite[Theorem 1.3.17]{Berndt-book} that connect Rogers' series (\ref{S79})-(\ref{Rogers-1}) with these characters (see also \cite[p.\ 170]{Kawasetsu}):
\begin{align}
Z_{1}(\tau)&:=q^{1/40} \sum_{n=0}^\infty \frac{q^{n^2+n}}{(q;q)_{2n}}={\rm ch}_{3,5}^{1,1}(\tau), \\
Z_{2}(\tau)&:=q^{31/40} \sum_{n=0}^\infty \frac{q^{n^2+2n}}{(q;q)_{2n+1}}={\rm ch}_{3,5}^{2,1}(\tau), \\
Z_{3}(\tau)&:=q^{9/40} \sum_{n=0}^\infty \frac{q^{n^2+n}}{(q;q)_{2n+1}}={\rm ch}_{3,5}^{2,2}(\tau),\\
Z_{4}(\tau)&:=q^{-1/40} \sum_{n=0}^\infty \frac{q^{n^2}}{(q;q)_{2n}}={\rm ch}_{3,5}^{1,2}(\tau).
\end{align}

These expressions transforms under $\mathrm{SL}(2,\mathbb{Z})$, and they define a $4$-dimensional vector-valued modular form $\rho_2$ of weight zero 
with a well-known $S$-matrix:
$$\begin{pmatrix} Z_1(-1/\tau) \\ Z_{2}(-1/\tau) \\ Z_{3}(-1/\tau) \\Z_{4}(-1/\tau) \end{pmatrix}=\sqrt{\frac{2}{5}} \begin{pmatrix} \sin(2\pi/5) & -\sin(2\pi/5) & - \sin(\pi/5) & \sin (\pi/5) \\
-\sin(2\pi/5) & -\sin(2\pi/5) & \sin(\pi/5) & \sin (\pi/5) \\ -\sin(\pi/5) & \sin(\pi/5) & -\sin(2\pi/5) & \sin (2\pi/5) \\ \sin(\pi/5) & \sin(\pi/5) & \sin(2\pi/5) & \sin (2\pi/5) \end{pmatrix}
\begin{pmatrix} Z_1(\tau) \\ Z_{2}(\tau) \\ Z_{3}(\tau) \\Z_{4}(\tau) \end{pmatrix}.$$
See for instance \cite[p.\ 172, Eq.\ (15)]{Kawasetsu} except for a typo in the $(3,3)$-entry of this matrix. 
One can easily write the diagonal $T$-matrices for $\rho_1$ and $\rho_2$ using the leading terms.
Observe that $\rho_1 \otimes \rho_2$, that is $(W_i Z_j)_{1 \leq i \leq 2,1 \leq j \leq 4}$, defines a $8$-dimensional vector-valued modular form on $\Gamma(1)$. However, we also have a
proper subspace in $\rho_1 \otimes \rho_2$ that also closes under the full modular group.
\begin{prop} \label{6-dim} Let $W_i$ and $Z_i$ be as above. We have 
\begin{align} 
&\Big({\tilde{F}_1(q)}, {\tilde{F}_2(q)}, {\tilde{F}_3(q)}, {\tilde{F}_4(q)}, {\tilde{F}_5(q)}, {\tilde{F}_6(q)}\Big) = (\frak{f}(\tau)(W_1 Z_4+W_2 Z_3), \frak{f}_2(\tau)(W_1 Z_1+W_2 Z_2),   \nonumber \\
& \frak{f}_2(\tau)(W_1 Z_3+W_2 Z_4), \frak{f}(\tau)(W_1 Z_2+W_2 Z_1), \frak{f}_1(\tau)(W_1 Z_4-W_2 Z_3),\frak{f}_1(\tau)(W_2 Z_1-W_1 Z_2)), \label{F-W-vect}
\end{align}

and it defines a vector-valued modular form $\tilde{\rho}_1$ of weight zero on $\Gamma(1)$.
\end{prop}
In fact, the equality \eqref{F-W-vect} follows directly from Theorem \ref{main2}.
One can also see this in a more straightforward way from the proof of Theorem \ref{main2}. For instance, from \eqref{1-F-split}, \eqref{1-S0-result} and \eqref{1-S1-result} (after replacing $q^2$ by $q$) we immediately conclude that
\begin{align}
    \tilde{F}_1(q)=\mathfrak{f}(\tau)(W_1Z_4+W_2Z_3).
\end{align}
The second assertion of this proposition follows by direct computations using the $S$ and $T$-matrices of $W_1, W_2$ and $Z_i$ ($1\leq i \leq 4$).

Denote by $\tilde{\rho}_{2}$ the vector-valued  modular form of weight zero:
\begin{align}
& (h_1(q),h_2(q),h_3(q),h_4(q),h_5(q),h_6(q)):=
\left(\frac{\frak{f}(\tau)^3}{\eta(\tau)^3} (\partial \Theta)_{1,\frac52}(\tau), 2\frac{\frak{f}_2(\tau)^3}{\eta(\tau)^3} (\partial \Theta)_{\frac12,\frac52}(\tau), \right. \nonumber \\ 
& 2\frac{\frak{f}_2(\tau)^3}{\eta(\tau)^3} (\partial \Theta)_{\frac32,\frac52}(\tau) , 
\label{second} \left.\frac{\frak{f}(\tau)^3}{\eta(\tau)^3} (\partial \Theta)_{2,\frac52}(\tau),
\frac{\frak{f}_1(\tau)^3}{\eta(\tau)^3} (\partial G)_{1,\frac52}(\tau),
\frac{\frak{f}_1(\tau)^3}{\eta(\tau)^3} (\partial G)_{2,\frac52}(\tau)\right).
\end{align}
This is precisely the vector-valued modular form appearing in Corollary \ref{gam}(2),  except that we added the factor 2 for the second and third entries.

So far we constructed two $6$-dimensional vector valued modular forms: $\tilde{\rho}_1$ in (\ref{F-W-vect}) and $\tilde{\rho}_2$  in (\ref{second}).
It is easy to see that (by direct computations) the $S$ and $T$ matrices of them agree with each other. We claim that  $\tilde{\rho}_1=\tilde{\rho}_2$.

To prove the above claim, we switch to new bases of vector-valued modular forms. Recall the basis $g_i(q)$ from Remark \ref{new-basis}. Similarly we let 
\begin{align*}
\tilde{g}_1(q)&:=h_1(q)+h_5(q)=q^{-7/80}(2+12q+30q^2+\cdots) \in q^{-7/80} \mathbb{C}[[q]], \\
\tilde{g}_2(q)&:=h_1(q)-h_5(q)=q^{33/80}(6+18q+54q^2+\cdots) \in q^{33/80}\mathbb{C}[[q]], \\
\tilde{g}_3(q)&:=h_4(q)+h_6(q)=q^{17/80}(4+6q+30q^2+\cdots) \in q^{17/80} \mathbb{C}[[q]], \\
\tilde{g}_4(q)&:=h_4(q)-h_6(q)=q^{57/80}(6+16q+42q^2+\cdots)  \in q^{57/80}\mathbb{C}[[q]],  \\
\tilde{g}_5(q)&:=h_2(q)=q^{1/40}(1+6q+15q^2+\cdots) \in q^{1/40} \mathbb{C}[[q]], \\ \tilde{g}_6(q)&:=h_3(q)=q^{9/40}(3+11q+30q^2+\cdots) \in q^{9/40}\mathbb{C}[[q]].
\end{align*}

Then $(g_i(q)-\tilde{g}_i(q))_{i=1}^6$ also transforms as a vector-valued modular form. Using their $q$-expansions it is easy to 
see (cf. Remark \ref{rem-wronski}) that the Wronskian of $g_i(h)-\tilde{g}_i(q)$, $1 \leq i \leq 6$
has the order of vanishing that is strictly bigger than $\frac{5}{2}$.
But then, according to Proposition \ref{ODE}, the Wronskian is identically zero and thus $(g_i(q)-\tilde{g}_i(q))$ are linearly dependent. The last sentence in Remark \ref{new-basis} implies that the linear dependence is equivalent to $g_j(q)=\tilde{g}_j(q)$ for some $j$. Thus we obtain a $5$-dimensional vector-valued modular form $(g_i(q)-\tilde{g}_i(q))_{i \neq j}$. Applying the same type of argument to it yields $g_k(q)=\tilde{g}_k(q)$ for $k \neq j$, etc. Thus we conclude $g_i(q)=\tilde{g}_i(q)$ for all $i$. That clearly implies $\tilde{F}_i(q)=h_i(q)$ and completes the proof of Theorem \ref{main}.
\begin{rem}
The claim $\tilde{\rho}_1(\tau)=\tilde{\rho}_2(\tau)$ can also be proved by a Sturm-type criterion of vector-valued modular forms. From \cite[Proposition 1.1]{MR} we know that it suffices to compare the first two terms in their $q$-expansions. Interestingly, in the above proof using Wronskian,  we also only need to compare the first two terms of $g_i(q)$ and $\tilde{g}_i(q)$, which guarantees that the Wronskian of $g_i(q)-\tilde{g}_i(q)$, $1 \leq i \leq 6$ has the order of vanishing $>\frac{5}{2}$. 
\end{rem}







\subsection{Second proof of Theorem \ref{main}}
We will need the order of a function $f(\tau)$ with respect to a congruence subgroup $G$ at the cusp $p\in \mathbb{Q}\cup \{\infty\}$ and denote it as $\mathrm{ord}(f,p)$.

Recall the theta function $\theta_3(\tau)$ defined in \eqref{theta3-prod}. 
It is known that \cite[Proposition 1.41]{Ono} $\theta_3(\tau)\in M_{\frac{1}{2}}(\Gamma_0(4))$. For any odd primitive Dirichlet character $\psi$ with conductor $N$, it is known that \cite[Theorem 1.44]{Ono}
$$\theta(\psi,\tau):=\sum_{n=1}^\infty \psi(n) nq^{n^2}\in S_{\frac{3}{2}}(\Gamma_0(4N^2,\psi\chi_{-4}),$$
where $\chi_{-4}$ is the nontrivial Dirichlet character modulo 4.

\begin{lemma}\label{lem-theta}
For $a\in \{1,2,3,4\}$, we have
\begin{align}
\sum_{\begin{smallmatrix} n\in \mathbb{Z} \\ n\equiv a ~~\mathrm{(mod~~5)} \end{smallmatrix}} nq^{n^2} \in M_{\frac{3}{2}}(\Gamma_1(100)).
\end{align}
\end{lemma}
\begin{proof}
Let $\psi_k$ ($k=0,1$) be the primitive Dirichlet character with conductor 5 satisfying $\psi_k(2)=(-1)^k i$. Then
\begin{align}
\theta(\psi_0,\tau)&=\sum_{n\in \mathbb{Z}} \psi_0(n)nq^{n^2}
= \sum_{n\in \mathbb{Z}}(5n+1)q^{(5n+1)^2}-\sum_{n\in \mathbb{Z}}(5n+4)q^{(5n+4)^2} \nonumber \\
&+i\left( \sum_{n\in \mathbb{Z}}(5n+2)q^{(5n+2)^2}-\sum_{n\in \mathbb{Z}}(5n+3)q^{(5n+3)^2} \right) \nonumber \\
&=2\sum_{n\in \mathbb{Z}} (5n+1)q^{(5n+1)^2} +2i\sum_{n\in \mathbb{Z}} (5n+2)q^{(5n+2)^2}. \label{psi1}
\end{align}
Here for the last equality we used the fact that for $a\in \{1,2\}$,
\begin{align}
\sum_{n\in \mathbb{Z}} (5n+a)q^{(5n+a)^2}=-\sum_{n\in \mathbb{Z}}(5n+5-a)q^{(5n+5-a)^2}, \label{theta-equal}
\end{align}
which can be proved easily by changing $n$ to $-n-1$.

In the same way, we have
\begin{align}
\theta(\psi_1,\tau)=2\sum_{n\in \mathbb{Z}} (5n+1)q^{(5n+1)^2}-2i\sum_{n\in \mathbb{Z}} (5n+2)q^{(5n+2)^2}. \label{psi2}
\end{align}
Since $\theta(\psi_k,\tau)\in M_{\frac{3}{2}}(\Gamma_1(100))$ ($k=0,1$), from \eqref{psi1} and \eqref{psi2} we deduce that
\begin{align}
\sum_{n\in \mathbb{Z}} (5n+1)q^{(5n+1)^2} &=\frac{1}{4}(\theta(\psi_0,\tau)+\theta(\psi_1,\tau))\in M_{\frac{3}{2}}(\Gamma_1(100)), \\
\sum_{n\in \mathbb{Z}} (5n+2)q^{(5n+2)^2} &=\frac{1}{4i}(\theta(\psi_0,\tau)-\theta(\psi_1,\tau))\in M_{\frac{3}{2}}(\Gamma_1(100)).
\end{align}
This together with \eqref{theta-equal} proves the lemma.
\end{proof}

We also need the following identity
\begin{align}
J_1^2=\frac{J_2J_8^5}{J_4^2J_{16}^2}-2q\frac{J_2J_{16}^2}{J_8}. \label{J1square}
\end{align}
This appeared frequently in the literature and is a direct consequence of \cite[p.\ 40, Entry 25(v),(vi)]{Notebook3}.

Following the notion in \cite{Garvan-Liang}, we define the generalized Dedekine eta function
\begin{align}\label{general-eta}
    \eta_{\delta;g}(\tau):=q^{\frac{1}{2}\delta P_2(g/\delta)}\prod\limits_{m\equiv \pm g \pmod{\delta}} (1-q^m),
\end{align}
where $P_2(t)=\{t\}^2-\{t\}+\frac{1}{6}$ is the second periodic Bernoulli polynomial, $\{t\}$ is the fractional part of $t$, $g,\delta, m\in \mathbb{Z}^+$ and $0<g<\delta$.

Now we are ready to give our second proof of Theorem \ref{main}. 
By the definition given in \eqref{F5-defn} and \eqref{F6-defn}, we know that the identities for $\widetilde{F}_5(q)$ and $\widetilde{F}_6(q)$ follow from those of $\widetilde{F}_1(q)$ and $\widetilde{F}_4(q)$, respectively. Thus, it suffices to prove the first four identities. By \eqref{sum-id-1}--\eqref{sum-id-4}, the first four identities are equivalent to
\begin{align}
T_1(q)&:=\sum_{n\in \mathbb{Z}}(5n+1)q^{5n^2+2n}=\frac{J_1^2J_4^8J_{40}}{J_2^5J_8^2J_{8,40}}+2q\frac{J_1^2J_4J_8^2J_{6,20}J_{8,40}}{J_2^3J_{40}}, \label{T1-id} \\
T_2(q)&:=\sum_{n\in \mathbb{Z}}(5n+2)q^{5n^2+4n}=2\frac{J_1^2J_4J_8^2J_{2,20}J_{16,40}}{J_2^3J_{40}}+q\frac{J_1^2J_4^7J_{8,20}J_{4,40}}{J_2^5J_8^2J_{40}},  \label{T2-id} \\
T_3(q)&:=\sum_{n\in \mathbb{Z}}(10n+1)q^{5n^2+n}=\frac{J_2J_4^3J_{2,20}J_{16,40}}{J_8^2J_{40}}+2q^2\frac{J_2^3J_8^2J_{8,20}J_{4,40}}{J_4^3J_{40}}, \label{T3-id} \\
T_4(q)&:=\sum_{n\in \mathbb{Z}}(10n+3)q^{5n^2+3n}=2\frac{J_2^3J_8^2J_{40}}{J_4^2J_{8,40}}+\frac{J_2J_4^3J_{6,20}J_{8,40}}{J_8^2J_{40}}. \label{T4-id}
\end{align}
We will prove \eqref{T1-id} first and then deduce the other identities from it. After replacing $q$ by $q^5$ and multiplying both sides of \eqref{T1-id} by $q\theta_3(5\tau)$, we see that \eqref{T1-id} is equivalent to
\begin{align}
\theta_3(5\tau)\sum_{n\in \mathbb{Z}}(5n+1)q^{(5n+1)^2}=\theta_3^4(5\tau)\left(f_1(\tau)+f_2(\tau) \right). \label{T1-id-equiv}
\end{align}
Here
\begin{align}
f_1(\tau)&:=q\frac{J_5^8J_{20}^{14}J_{200}}{J_{10}^{20}J_{40}^2J_{40,200}}=\frac{\eta^8(5\tau)\eta^{14}(20\tau)}{\eta^{20}(10\tau)\eta^2(40\tau)\eta_{200,40}(\tau)}, \\
f_2(\tau)&:=2q^6\frac{J_5^8J_{20}^7J_{40}^2J_{30,100}J_{40,200}}{J_{10}^{18}J_{200}} \nonumber \\
&=\frac{\eta^8(5\tau)\eta^7(20\tau)\eta^2(40\tau)\eta(100\tau)\eta_{100,30}(\tau)\eta_{200,40}(\tau)}{\eta^{18}(10\tau)}.
\end{align}

With the help of Maple and the algorithm in \cite{Garvan-Liang}, it is easy to check that both $f_1(\tau)$ and $f_2(\tau)$ are modular functions on $\Gamma_1(200)$. Their poles and corresponding orders of them are listed in Table \ref{tab-ord}.
\begin{table}[h]
 \renewcommand\arraystretch{1.5}
\begin{tabular}{c|c|c}
  \hline
   cusp $p$ & $\mathrm{ord}(f_i(\tau),p)$ & $\mathrm{ord}(\theta_3^4(5\tau),p)$ \\
   \hline
$\begin{matrix}\frac{1}{2} , \frac{1}{6} ,  \frac{1}{14} ,  \frac{1}{18}, \frac{1}{22}, \frac{1}{26}, \frac{1}{34},
\frac{1}{38}, \frac{1}{42}, \frac{1}{46}, \\
\frac{1}{54}, \frac{1}{58}, \frac{1}{62}, \frac{1}{66},
\frac{1}{74}, \frac{1}{78}, \frac{1}{82}, \frac{1}{86}, \frac{1}{94}, \frac{1}{98} \end{matrix}$  & $-16$ & $20$ \\
\hline
$ \begin{matrix} \frac{1}{10},  \frac{1}{30},  \frac{1}{70}, \frac{1}{90}, \frac{3}{10}, \frac{3}{70}, \frac{7}{10}, \frac{7}{30}, \\
\frac{7}{90}, \frac{9}{10}, \frac{9}{70}, \frac{23}{30},  \frac{23}{90},  \frac{29}{30}, \frac{29}{90}, \frac{67}{70} \end{matrix}$  & $-80$ & $100$ \\
\hline
 $\frac{1}{50}, \frac{9}{50}, \frac{11}{50}, \frac{19}{50}, \frac{21}{50},  \frac{29}{50},  \frac{31}{50}, \frac{39}{50}, \frac{41}{50}, \frac{49}{50}$   & $-16$  & $20$\\
\hline
$\frac{3}{50},  \frac{7}{50}, \frac{13}{50}, \frac{17}{50}, \frac{23}{50}, \frac{27}{50}, \frac{33}{50}, \frac{37}{50}, \frac{43}{50}, \frac{47}{50}$ & $-14$ & $20$\\
  \hline
\end{tabular}
\vspace{2mm}
\caption{Orders of poles and zeros at cusps for $\Gamma_1(200)$}
\label{tab-ord}
\end{table}

On the other hand, it is easy to see that $\theta_3^4(5\tau) \in M_2(\Gamma_0(20))$. Hence $\theta_3^4(5\tau)\in M_2(\Gamma_1(200))$. We can evaluate the orders of zeros of $\theta_3^4(5\tau)$ at any cusp for $\Gamma_1(200)$ (using Theorem 2.3 and the equation (2.12) in \cite{Garvan-Liang}). In Table \ref{tab-ord} we only list the orders of zeros of $\theta_3^4(5\tau)$ at the poles of $f_i(\tau)$ ($i=1,2$). It turns out that after multiplying by $\theta_3^4(5\tau)$, all the poles of $f_i(\tau)$ will be eliminated. Hence $\theta_3^4(5\tau) f_i(\tau)\in M_2(\Gamma_1(200))$ ($i=1,2$).

So far we have proved that both sides of \eqref{T1-id-equiv} belong to $M_2(\Gamma_1(200))$.  By Sturm's criterion (see \cite[p.\ 185, Corollary 5.6.14]{Cohen-Stromberg}, for example), to prove \eqref{T1-id-equiv}, it suffices to verify that both sides agree for the first
$$1+\frac{2}{12}\cdot \frac{1}{2}[\mathrm{SL}(2,\mathbb{Z}):\Gamma_1(200)]=2401$$
terms. We have checked this with Maple. Hence \eqref{T1-id-equiv} holds and we finish the proof of \eqref{T1-id}.

Next, we are going to prove \eqref{T2-id}--\eqref{T4-id} based on \eqref{T1-id}. We aim to find a 2-dissection formula for $T_1(q)$:
\begin{align}
T_1(q)=\sum_{n\in \mathbb{Z}}(5n+1)q^{5n^2+2n}=L_0(q^2)+qL_1(q^2). \label{T1-dissectin}
\end{align}
On the one hand, since $5n^2+2n$ has the same parity with $n$, we have
\begin{align}
L_0(q^2)&=\sum_{n ~\text{even}}(5n+1)q^{5n^2+2n}=\sum_{n\in \mathbb{Z}}(10n+1)q^{20n^2+4n}=T_3(q^4), \label{L0-T} \\
qL_1(q^2)&=\sum_{n~ \text{odd}} (5n+1)q^{5n^2+2n}=\sum_{n\in \mathbb{Z}}(5(-2n-1)+1)q^{5(-2n-1)^2+2(-2n-1)}\nonumber \\
&=-2q^3\sum_{n\in \mathbb{Z}}(5n+2)q^{20n^2+16n}=-2q^3T_2(q^4). \label{L1-T}
\end{align}
On the other hand, substituting \eqref{J1square} into \eqref{T1-id}, we obtain
\begin{align}
&\sum_{n\in \mathbb{Z}}(5n+1)q^{5n^2+2n}=\left(\frac{J_2J_8^5}{J_4^2J_{16}^2}-2q\frac{J_2J_{16}^2}{J_8} \right)\cdot \left(\frac{J_4^8J_{40}}{J_2^5J_8^2J_{8,40}}+2q\frac{J_4J_8^2J_{6,20}J_{8,40}}{J_2^3J_{40}}\right) \nonumber \\
&=\frac{J_4^6J_8^3J_{40}}{J_2^4J_{16}^2J_{8,40}}-4q^2\frac{J_4J_8J_{16}^2J_{6,20}J_{2,40}}{J_2^2J_{40}} +2q\left(\frac{J_8^7J_{6,20}J_{8,40}}{J_2^2J_4J_{16}^2J_{40}}-\frac{J_4^8J_{16}^2J_{40}}{J_2^4J_8^3J_{8,40}} \right). \label{T1-final}
\end{align}
Combining \eqref{T1-dissectin}--\eqref{T1-final}, we deduce that
\begin{align}
&L_0(q)=T_3(q^2)=\frac{J_2^6J_4^3J_{20}}{J_1^4J_8^2J_{4,20}}-4q\frac{J_2J_4J_8^2J_{3,10}J_{4,20}}{J_1^2J_{20}}, \label{F0-result} \\
&L_1(q)=-2qT_2(q^2)=\frac{J_4^7J_{3,10}J_{4,20}}{J_1^2J_2J_8^2J_{20}}-\frac{J_2^8J_8^2J_{20}}{J_1^4J_4^3J_{4,20}}. \label{F1-result}
\end{align}
Using \eqref{F0-result} and \eqref{F1-result} and the method in \cite{Garvan-Liang}, it is easy to verify that \eqref{T2-id} and \eqref{T3-id} hold.

It remains to prove \eqref{T4-id}. For this we make a 2-dissection for $T_2(q)$:
\begin{align}
T_2(q)=\sum_{n\in \mathbb{Z}}(5n+2)q^{5n^2+4n}=H_0(q^2)+qH_1(q^2).
\end{align}
In the same way as we did for $T_1(q)$, we can prove that
\begin{align}
&H_0(q)=2T_1(q^2)=2\frac{J_4^7J_{1,10}J_{8,20}}{J_1^2J_2J_8^2J_{20}}-2q\frac{J_2^7J_8^2J_{4,10}J_{2,20}}{J_1^4J_4^3J_{20}}, \label{H0-result} \\
&H_1(q)=-T_4(q^2)=\frac{J_2^5J_4^3J_{4,10}J_{2,20}}{J_1^4J_8^2J_{20}}-4\frac{J_2J_4J_8^2J_{1,10}J_{8,20}}{J_1^2J_{20}}. \label{H1-result}
\end{align}
Now from \eqref{H1-result} and using the method in \cite{Garvan-Liang}, it is easy to verify that \eqref{T4-id} holds.

\section{General conjecture and concluding remarks}\label{sec-general}

\subsection{Non-modularity of some Nahm sums} 
In \cite[Section 5]{CMP} in addition to $\chi_0(1,1,1)$ and $\chi_0(1,1,q^{\frac12})$ two additional specializations are considered: $ \chi_0(q,1,1)$ and $ \chi_0(1,q,1)$. Thus it 
seems natural to ask whether these two series are also modular after addition of a suitable multiplicative factor. We next show that this is not the case. 

Recall that we denote by $f_{A,B,C}(q)$ the Nahm sum associated to $T_3$ matrix with $B=(B_1,B_2,B_3)$ as in Section 1.
Using this notation, we can write $f_{A,(1,0,0),C}(q)=q^C \chi_0(q,1,1)$  and $f_{A,(0,1,1),C}(q)=q^C \chi_0(1,q,1)$. Denote by 
 $$Q_1=\frac12(3-\sqrt{5}), \quad Q_2=-2+\sqrt{5}, \quad Q_3=\frac14(3-\sqrt{5})$$
the unique solution inside the interval $(0,1)$ of the TBA system \cite[Lemma 2.1]{VZ}: 
$$1-Q_1=Q_1^2 Q_2^{-1}, \quad 1-Q_2=Q_1^{-1} Q_2^2 Q_3^{-1}, \quad 1-Q_3=Q_3 Q_2^{-1}.$$
\begin{prop}  The Nahm sums $f_{A,(1,0,0),C}(q)$  and $f_{A,(0,1,0),C}(q)$ are not modular for any rational number $C$.
\end{prop}
\begin{proof} To prove this, it suffices to argue that the two Nahm sums do not have expected asymptotic expansion around zero. Letting $q=e^{-\epsilon}$, then according to \cite[Theorem 2.1]{VZ} we have the following asymptotic behavior (as $\epsilon \to 0^+$):
$$f_{A,B,C}(e^{-\epsilon})e^{-\frac{\alpha}{\epsilon}} \sim \beta e^{-\gamma \epsilon}(1+\sum_{p \geq 1} c_p \epsilon^p)$$
where $\alpha$ is a positive constant, $\gamma=C+\frac{1}{24} \sum_{i=1}^r \frac{1+Q_i}{1-Q_i}$, $c_p$ are some hard-to-compute coefficients expressed using generalized $3$-fold Gaussian integrals and $\beta$ is a nonzero constant not needed here. As a necessary condition for modularity we notice relations \cite[Corollary 3.1]{VZ}:
$$\frac{\gamma^p}{p!} -c_p=0, \ p \geq 1.$$
In our situation, the condition $c_1-\gamma=0$ is equivalent to 
\begin{align*}
C&=\frac{9 {B_1}^2}{4 \sqrt{5}}-\frac{3 {B_1}^2}{4}+\frac{3 {B_1} {B_2}}{\sqrt{5}}
-{B_1}{B_2}-\frac{ {B_1}  {B_3}}{2 \sqrt{5}}+\frac{ {B_1}  {B_3}}{2}+\frac{2  {B_1}}{\sqrt{5}}-\frac{9
    {B_1}}{10} \\
    & \quad +\frac{ {B_2}^2}{\sqrt{5}}+\frac{ {B_2}  {B_3}}{2 \sqrt{5}}+\frac{ {B_2}
    {B_3}}{2}+\frac{7  {B_2}}{4 \sqrt{5}}-\frac{17
    {B_2}}{20}+\frac{ {B_3}^2}{\sqrt{5}}+\frac{ {B_3}^2}{4}-\frac{ {B_3}}{2
   \sqrt{5}}+\frac{ {B_3}}{10}-\frac{7}{80}.
\end{align*}
Plugging in $B=(1,0,0)$ gives the value $C= \frac{1}{80}(-139 + 68 \sqrt{5})$ which is irrational and similarly for $B=(0,1,0)$.
Therefore the two Nahm sums in question  cannot be modular.
\end{proof}

Although $\chi_0(q,1,1)$ and $\chi_0(1,q,1)$ cannot be made modular their sum satisfies \begin{align}\label{add-formula}
\chi_0(q,1,1)+\chi_0(1,q,1)=\chi_0(q^{-1},q,1)
\end{align}
which is modular after 
multiplying with $q^{\frac{17}{80}}$. Relation (\ref{add-formula}) follows
from a slightly more general statement:
\begin{align} \label{x-var}
\sum_{i,j,k \geq 0} \frac{q^{i^2+j^2+k^2/2-ij-jk+j-i}x_1^i x_2^j x_3^k}{(q)_i (q)_j(q)_k}=
\sum_{i,j,k \geq 0} \frac{q^{i^2+j^2+k^2/2-ij-jk}x_1^i x_2^j x_3^k(x_1 q^i+q^j)}{(q)_i (q)_j(q)_k}.
\end{align}
Comparing the coefficients of  $x_1^i x_2^j x_3^k$ of both sides of (\ref{x-var}), we see that (\ref{x-var}) is equivalent to
\begin{align*}
\frac{q^{i^2+j^2+k^2/2-ij-jk+j-i}}{(q)_i (q)_j(q)_k}&=\frac{q^{i^2+j^2+k^2/2-ij-jk+j}}{(q)_i (q)_j(q)_k}+
\frac{q^{(i-1)^2+j^2+k^2/2-(i-1)j-jk+(i-1)}}{(q)_{i-1} (q)_j(q)_k}.
\end{align*}
The last formula is trivial to check.

\subsection{General conjecture}
In this part we discuss the general conjecture on the modularity of the rank $n$ Nahm sum $\chi_0({\bf 1})$ associated to $T_n$ as in the introduction, where for brevity we let ${\bf 1}:=(1,1,...,1)$.
It is possible to formulate a slightly stronger conjecture result analogous to Theorem \ref{main} but we omit discussing it here.

Let $f_1$,...,$f_\ell$ be any holomorphic functions in the upper half-plane. Denote by $D=\left(q \frac{d}{dq}\right)=\frac{1}{2 \pi i } \frac{\partial}{\partial \tau}$ Ramanujan's derivative and by $$\partial_k:=D-\frac{k}{12}E_2$$ where $E_2(q)=1-24\sum_{n \geq 1} \frac{nq^n}{1-q^n}$ is the second Eisenstein
series. This map is known to send the space of modular forms of weight $k$ (on some congruence subgroup) into the space of modular forms of weight $k+2$ for the same congruence subgroup. Denote by $\mathcal{W}_D(f_1,...,f_\ell)$ the Wronskian determinant with respect to the $D$-derivation which is again a holomorphic function. Additionally, denote by
$\mathcal{W}_{\partial_k}(f_1,...,f_\ell)$ the Wronskian with respect to the $\partial_k$ derivation, where the $r$-th derivative is defined as
$\partial^r_k:=\partial_{k+2r-2} \circ \cdots \circ \partial_k$. Suppose that each $f_i$ admits a $q$-expansion. Then $\mathcal{W}_D$ also has a $q$-expansion so we can denote by $\widetilde{\mathcal{W}}_D$ the Wronskian normalized such that the leading coefficient in the $q$-expansion is $1$. Then we have a known result (see for instance \cite{Bringmann, Milas}).
\begin{theorem} \label{wronski} Let $f_1$,\dots,$f_\ell$ be modular forms of weight $k$ with respect to a congruence subgroup, then the Wronskian
$$\mathcal{W}_D(f_1,...,f_\ell)=\mathcal{W}_{\partial_k}(f_1,...,f_\ell)$$
is a modular form of weight $\ell(\ell+k-1)$ on the same congruence subgroup.
\end{theorem}
In \cite{Bringmann} a more precise information about the automorphy factor and congruence subgroups for the Wronskian modular form $\mathcal{W}_D$ is given.

Now we specialize Theorem \ref{wronski} to a situation where $V={\rm Span}(f_1,...,f_\ell)$ define an $\ell$-dimensional modular invariant space
under $\Gamma(1)$. Then we have the following result.
\begin{prop} \label{ODE} 
(1) Let $f_1$,..,$f_\ell$ be a basis of the modular invariant space $V$. If  
$${\rm ord} (\mathcal{W}_D(f_1,...,f_\ell),i \infty)=\lambda$$
then 
$${\mathcal{W}}_D(f_1,...,f_\ell)=\eta(\tau)^{24 \lambda} G(\tau),$$
where $G(\tau)$ is a nonzero holomorphic modular form of weight $\ell(\ell-1)-12 \lambda$ on $\Gamma(1)$. In particular, if $\lambda=\frac{\ell(\ell-1)}{12}$, 
then $$\widetilde{\mathcal{W}}_D(f_1,...,f_\ell)=\eta(\tau)^{2\ell(\ell-1)}.$$
(2) If $${\rm ord} (\mathcal{W}_D(f_1,...,f_\ell),i \infty)>\frac{\ell(\ell-1)}{12}$$
then the Wronskian is identically zero and $f_i$ are linearly dependent.
\end{prop}
Part (1) of this proposition can be found in \cite[Theorem 3.7]{Mason}; see also \cite[Theorem 2.2 and Proposition 2.4]{Milas} and \cite[Theorem 1]{Nagatomo}. Part (2) follows from part (1) and the fact that $M_k(\Gamma(1))=0$ for $k<0$. See also \cite[Lemma 3.6]{Mason}.

\begin{rem} \label{rem-wronski} In the situation when $f_i(q)=q^{r_i}(a_0^{(i)}+ a_1^{(i)}q+ \cdots)$, $a_0^{(i)} \neq 0$ for every $i$ and $r_i \neq r_j$ for $i \neq j$, it is easy to see that the order of vanishing of $\mathcal{W}_D(f_1,...,f_\ell)$ is $\sum_{i=1}^\ell r_i$.
Using this fact and Remark \ref{new-basis} we immediately see that the order of vanishing of $\mathcal{W}_{D}(\tilde{F}_1,\tilde{F}_2,\tilde{F}_3,\tilde{F}_4,\tilde{F}_5,\tilde{F}_6)$ is $\frac32$, so in our case 
the Wronskian is not an $\eta$-power and instead we have
$$\widetilde{\mathcal{W}}_D(\tilde{F}_1,\tilde{F}_2,\tilde{F}_3,\tilde{F}_4,\tilde{F}_5,\tilde{F}_6)=(5892480)^{-1} \cdot \eta(\tau)^{36} (70 027513 E_4(\tau)^3 - 64135033 E_6(\tau)^2).$$
\end{rem}

Now we are ready to state a general conjecture which is based on some numerical evidence.
\begin{conj} Let $\chi_0({\bf 1})$ be the Nahm sum associated to $T_n$, $n \geq 2$. Then we have
\begin{enumerate}
\item For $n=2k \geq 2$ even:
$$q^{a_k} \chi_0({\bf 1})=\frac{\frak{f}(\tau)^{2k} \widetilde{\mathcal{W}}_D(R_{2k,1},....,R_{2k,k}) }{\eta(\tau)^{k(2k-1)}},$$
where  $a_k=-\frac{k(1+4k)}{48(1+k)}$ and $$R_{2k,i}(\tau)=\sum_{n \in \mathbb{Z}} (-1)^n q^{(k+1)(n-\frac{(2i-1)}{4(k+1)})^2}.$$
\item For $n=2k-1 \geq 3$ odd:
$$q^{a_k} \chi_0({\bf 1})= \frac{\frak{f}(\tau)^{2k-1} \widetilde{\mathcal{W}}_D((\partial \Theta)_{1,\frac{2k+1}{2}},....,(\partial \Theta)_{k-1,\frac{2k+1}{2}})}{\eta(\tau)^{(k-1)(2k-1)}},$$
where $a_k=\frac{-1+6k-8k^2}{96k+48}$.
\end{enumerate}
More precisely, $q^{a_k} \chi_0({\bf 1})$ is a modular function which is a component of a $3k$-dimensional vector valued modular function under $\Gamma(1)$.
\end{conj}

Part (1) of the conjecture is known to hold for $n=2$  \cite{CMP} and we also verified $n=4$ numerically for high powers of $q$.
Part (2) for $n=3$ was proven in this paper.

\subsection*{Acknowledgements}
The second author was supported by the National Natural Science Foundation of China (12171375).

\end{document}